\documentclass[a4paper, 12pt, oneside]{article}
\pdfoutput=1 
\usepackage[T1]{fontenc}
\usepackage[british]{babel}
\usepackage{lmodern}
\usepackage{etoolbox}
\usepackage{amsmath, amsthm, amssymb, bm, mleftright}
\usepackage{booktabs, longtable, caption, subcaption, multirow}
\usepackage{scalefnt}
\usepackage[space]{grffile}
\usepackage[margin=2.5cm, footskip=1cm]{geometry}
\usepackage{enumitem}
\setenumerate[1]{label=\textup{(\arabic*)}, ref=\arabic*}
\setenumerate[2]{label=\textup{(\alph*)}, ref=\alph*}
\setenumerate[3]{label=\textup{(\roman*)}, ref=\roman*}
\usepackage[protrusion=allmath]{microtype}
\usepackage{tikz}
\usetikzlibrary{cd, calc}
\usepackage{listings}
\usepackage{colortbl}
\usepackage[all]{xy}
\usepackage{mathtools}
\usepackage[colorlinks, citecolor=blue, linkcolor=blue, linktoc=section]{hyperref}

\allowdisplaybreaks

\usepackage{titlesec}
\titleformat*{\section}{\large\bfseries}
\titleformat*{\subsection}{\normalsize\bfseries}
\newlength{\VerticalSpaceAfterParagraph}
\titlespacing*{\paragraph}{0pt}{\VerticalSpaceAfterParagraph}{1em}

\usepackage{titling}
\pretitle{\vspace{-\baselineskip}\begin{center}\Large\bfseries}
\posttitle{\end{center}\vspace{-0.25\baselineskip}}
\preauthor{\begin{center}}
\postauthor{\end{center}}
\predate{\begin{center}}
\postdate{\end{center}}

\setlist
  {
    topsep = 5.0pt plus 2.0pt minus 3.0pt,
    partopsep = 1.5pt plus 1.0pt minus 1.0pt,
    parsep = 2.5pt plus 1.25pt minus 0.5pt,
    itemsep = 0pt plus 1.25pt minus 0.5pt
  }

\theoremstyle{plain}
\newtheorem{theorem}{Theorem}
\newtheorem{proposition}[theorem]{Proposition}
\newtheorem{lemma}[theorem]{Lemma}
\newtheorem{corollary}[theorem]{Corollary}
\newtheorem{question}[theorem]{Question}

\theoremstyle{definition}
\newtheorem{definition}[theorem]{Definition}
\newtheorem{setting}[theorem]{Setting}
\newtheorem{condition}[theorem]{Condition}

\theoremstyle{remark}
\newtheorem{remark}[theorem]{Remark}

\numberwithin{theorem}{section}
\numberwithin{equation}{section}

\usepgfmodule{oo}

\DeclareRobustCommand\ShowAuthors[2]{%
  \ShowAuthorsSignal.emit({#1},{#2})%
}

\pgfoonew \ShowAuthorsSignal=new signal()

\DeclareRobustCommand\ShowAffiliations[1]{%
  \ShowAffiliationsSignal.emit(#1)%
}

\pgfoonew \ShowAffiliationsSignal=new signal()

\newcommand\Author[1]{
  \pgfoonew \CurrentPerson=new person()
  \CurrentPerson.set author(#1)
}

\newcommand\Email[1]{
  \CurrentPerson.set email(#1)
}

\newcommand\Address[1]{
  \CurrentPerson.set address(#1)
}

\newcommand\FirstPerson{0}
\newcommand\LastPerson{}

\pgfooclass{person}{
  \attribute author;
  \attribute email;
  \attribute address;

  \method person() {
    \pgfoothis.get id(\theid)

    \ifnum \FirstPerson=0
      \edef\FirstPerson{\theid}
    \fi

    \edef\LastPerson{\theid}

    \pgfoothis.get handle(\me)
    \ShowAuthorsSignal.connect(\me,show author)
    \ShowAffiliationsSignal.connect(\me,show affiliation)
  }

  \method get author(#1) {
    \pgfooget{author}{#1}
  }

  \method set author(#1) {
    \pgfooset{author}{#1}
  }

  \method get email(#1) {
    \pgfooget{email}{#1}
  }

  \method set email(#1) {
    \pgfooset{email}{#1}
  }

  \method get address(#1) {
    \pgfooget{address}{#1}
  }

  \method set address(#1) {
    \pgfooset{address}{#1}
  }

  \method show author(#1,#2) {%
    \pgfoothis.get id(\theid)%
    \pgfooget{author}{\theauthor}%
    \ifnum\theid=\FirstPerson%
    \else%
      \ifnum\theid=\LastPerson%
        #2%
      \else%
        #1%
      \fi%
    \fi%
    \mbox{\theauthor}%
  }

  \method show affiliation(#1) {%
    \pgfoothis.get id(\theid)%
    \pgfooget{author}{\theauthor}%
    \pgfooget{email}{\theemail}%
    \pgfooget{address}{\theaddress}%
    \ifnum\theid=\FirstPerson%
    \else%
      #1%
    \fi%
    \noindent\begin{minipage}{\linewidth}
      \noindent\begin{tabular}[t]{@{}l}
        \theauthor \quad \texttt{\theemail}\\
      \end{tabular}\\
      \theaddress%
    \end{minipage}%
  }
}

\newcommand\blankfootnote[1]
  {%
    \begin{NoHyper}%
      \renewcommand\thefootnote{}%
      \footnote{#1}%
      \addtocounter{footnote}{-1}%
    \end{NoHyper}%
  }

\Author{Igor Krylov}
\Address{Institute for Basic Science, 79 Jigok-ro 127beon-gil, Nam-gu, Pohang-si, Gyeongsangbuk-do, Republic of Korea, 37673}
\Email{ikrylov@ibs.re.kr}

\Author{Takuzo Okada}
\Address{Faculty of Mathematics, Kyushu University, Fukuoka 819-0395, Japan}
\Email{tokada@math.kyushu-u.ac.jp}

\Author{Erik Paemurru}
\Address{Mathematik und Informatik, Universität des Saarlandes, 66123 Saarbrücken, Germany.\\
\textit{Current address:} Institute of Mathematics and Informatics, Bulgarian Academy of Sciences, Acad.\ G.~Bonchev Str.\ bl.~8, 1113, Sofia, Bulgaria.}
\Email{algebraic.geometry@runbox.com}

\title{Local inequalities for $cA_k$ singularities}
\author{\ShowAuthors{, }{, }}
\date{29th~August 2025}
\newcommand\keywords{Fano varieties, birational rigidity, rationality, terminal, compound Du Val}
\newcommand\subjclass{14J45, 14E08, 14J30, 14C17, 14M10}

\begin{document}

\maketitle

\begin{abstract}
We generalize an intersection-theoretic local inequality of Fulton–Lazarsfeld to weighted blowups. As a consequence, we obtain the $4n^2/(k+1)$-inequality for isolated $cA_k$ singularities, an analogue of the $4 n^2$-inequality for smooth points. We use this to prove birational rigidity of many families of Fano 3-fold weighted complete intersections with terminal quotient singularities and isolated $cA_k$ singularities, including sextic double solids with $cA_1$ and ordinary $cA_2$ points.
\blankfootnote{\textup{2020} \textit{Mathematics Subject Classification}. \subjclass{}.}%
\blankfootnote{\textit{Keywords}. \keywords{}.}
\end{abstract}

\tableofcontents

\section{Introduction}

The notion of birational rigidity was originally introduced in the context of proving non-rationality of Fano varieties.
It is also crucial for the classification of Fano varieties up to birational equivalence.

\begin{definition}
    We say that a Fano variety is \textit{birationally rigid} if it is the unique Mori fiber space in its birational class.
    We say that a birationally rigid Fano variety $X$ is \textit{superrigid} if $\operatorname{Aut}(X) = \operatorname{Bir}(X)$.
\end{definition}

Birational rigidity is typically proven using the Noether--Fano method: non-rigidity of a Fano variety $X$ implies the existence of a mobile linear system $\mathcal M \subset \left| -nK_X \right|$ such that the pair $(X, \frac{1}{n} \mathcal M)$ has worse than canonical singularities.
Thus, interpreting non-canonicity of the pair $(X, \frac{1}{n} \mathcal M)$ in geometric terms is the key to being able to prove the birational rigidity of a Fano variety.
The $4n^2$-inequality is one of the most important ingredients for proof of birational rigidity of Fano varieties.

\begin{theorem}[{$4 n^2$-inequality, \cite[Theorem 2.1]{PukBook}}]
\label{thm:4nineq}
Let $P \in X$ be the germ of a smooth $3$-fold.
Let $\mathcal M$ be a mobile linear system on $X$ and let $n > 0$ be a rational number.
If $P$ is a center of non-canonical singularities of the pair~$(X, \frac{1}{n} \mathcal M)$, then for general members $D_1, D_2 $ in $\mathcal M$ we have
\[
\operatorname{mult}_{P} (D_1 \cdot D_2) > 4 n^2.
\]
\end{theorem}

Quasismooth Fano weighted hypersurfaces of index $1$ are proven to be birationally rigid using the $4n^2$-inequality and the uniqueness of divisorial contractions with center a cyclic quotient singularity (\cite{CP17}, \cite{CPR}).

In this paper we prove an analogue of Theorem \ref{thm:4nineq} for compound Du Val singularities of type $cA_k$.

\newcommand\maintheorem{%
    Let $P \in X$ be an isolated $cA_k$ singularity, where $k \ge 1$, let $\mathcal M$ be a mobile linear system of Cartier divisors and $n > 0$ a rational number.
    Suppose that $P$ is a non-canonical center of the pair $(X, \frac{1}{n} \mathcal M)$.
    Then for general members $D_1, D_2 \in \mathcal{M}$, we have
        \[
        \operatorname{mult}_P (D_1 \cdot D_2) > \frac{4}{k+1} n^2.
        \]
}

\begin{theorem}[= Theorem~\ref{body-MainThm}] \label{MainThm}
\maintheorem{}
\end{theorem}

As an application, we consider the birational rigidity of sextic double solids and also other weighted complete intersection index $1$ Fano varieties.

\begin{theorem}[{$=$ Theorem~\ref{thm:brWCI}}] \label{mainthm:WCI}
Let $X$ be a prime Fano $3$-fold weighted complete intersection which belongs to one of the families listed in Tables~\ref{table:Fanohyp} and~\ref{table:FanoWCI}.
Suppose that $X$ is quasismooth along the singular locus of the ambient weighted projective space, and $X$ has only $cA_k$ singularities besides terminal quotient singular points.
Suppose further that $k$ is not greater than the bound $k_{\mathrm{cA}}$ listed in the table.
Then $X$ is birationally rigid.
\end{theorem}

\subsection{\texorpdfstring
    {Approach for local inequalities for $cA_k$-points}
    {Approach for local inequalities for cA\_k-points}
}

Theorem~\ref{MainThm} is a generalization of \cite[Theorem~1.2]{KOPP24} but we achieve it with different methods.
In \cite{KOPP24} we considered a tower of blowups over $P \in X$ and we tracked the intersections on every level of the tower.
The key ingredient for this approach is the following result.

\begin{theorem} [{\cite[Theorem~3]{FL82}}]
Let $P$ be a $\Bbbk$-rational point of a $d$-dimensional smooth integral algebraic \mbox{$\Bbbk$-scheme}~$X$, where $\Bbbk$ is a field and $d \geq 2$.
Let $\varphi \colon Y \to X$ be the blow up at $P$ and let $E$ be the exceptional prime divisor.
Let $D_1, \ldots, D_d$ be effective $\mathbb{Q}$-divisors on $X$ such that $P$ is an isolated point of $D_1 \cap \ldots \cap D_d$.
For every~$i$, let $\tilde D_i$ be the strict transform of~$D_i$.
Then, both of the following hold:
\begin{align}
    \operatorname{mult}_P(D_1 \cdot \ldots \cdot D_d) & = v_E(D_1) \cdot \ldots \cdot v_E(D_d) + \deg_E(\tilde D_1 \cdot \ldots \cdot \tilde D_d),\\
    \deg_E(\tilde D_1 \cdot \ldots \cdot \tilde D_d) & \geq \deg_E(\tilde D_1 \cap \ldots \cap \tilde D_d),
\end{align}
where $\deg_E(\tilde D_1 \cdot \ldots \cdot \tilde D_d)$ denotes the degree of the $0$-cycle class $\tilde D_1 \cdot \ldots \cdot \tilde D_d$ with coefficients in $\mathbb Q$ on a small neighborhood of $E$.
Moreover, $\deg_E(\tilde D_1 \cdot \ldots \cdot \tilde D_d)$ is zero if and only if $\tilde D_1 \cap \ldots \cap \tilde D_d \cap E$ is empty.
\end{theorem}

Given a variety $X$ of dimension $3$ with a $cA_1$ singularity $P$ embedded in a $4$-dimensional variety $Y$ such that $(P \in Y)$ is smooth and two Cartier divisors $D_1,D_2$ on $X$, we can apply this theorem to $X, D_1^\prime, D_2^\prime, H$, where $D_i^\prime$ is some extension of $D_i$ to $Y$ and $H$ is a general very ample divisor passing through $P$.
This provides a way to estimate intersection multiplicities using multiplicities of divisors at points which allowed us to prove \cite[Theorem~1.2]{KOPP24}.

It was possible since on every level we had a factorial variety.
For $cA_2$-singularities, the exceptional divisor of the blow up is the sum of two non-Cartier divisors, and if we want to use the similar approach, then we have to consider factorialization of every level of the tower of blow ups.
In this case the computations become very cumbersome, so we opted for a different approach.

By Theorem~\ref{thm:cldccA2} (\cite[Theorem 1.13]{Kaw03}, \cite[Theorem 1.1]{Kaw02} and \cite[Theorem 2.6]{Yam18}), we know that every divisorial contraction to a $cA_k$ point is a weighted blowup with weights from a specific list.
This allows us to consider intersections for a single weighted blowup instead of tracking intersections through a tower of blowups.
On the other hand, in order to make this work we have to generalize \cite[Theorem~3]{FL82} to weighted blowups.

\newcommand\GeneralizationOfFulton[2]{%
    Let $P$ be a closed point of a $d$-dimensional regular integral algebraic \mbox{$\Bbbk$-scheme}~$X$, where $\Bbbk$ is a field and $d \geq 2$.
    Let ${\bm w} = (w_1, \ldots, w_d)$ be a tuple of positive integers such that $\gcd {\bm w} = 1$.
    Let $\varphi \colon Y \to X$ be the ${\bm w}$-blowup with respect to a regular system of parameters at $P$ and let $E$ be the exceptional prime divisor.
    Let $D_1, \ldots, D_d$ be effective $\mathbb{Q}$-divisors on $X$ such that $P$ is an isolated point of $\operatorname{Supp} D_1 \cap \ldots \cap \operatorname{Supp} D_d$.
    For every~$i$, let $\tilde D_i$ be the strict transform of~$D_i$.
    Then, both of the following hold:
    \begin{align}
        #1 \operatorname{mult}_P(D_1 \cdot \ldots \cdot D_d) & = \frac{v_E(D_1) \cdot \ldots \cdot v_E(D_d)}{w_1 \cdot \ldots \cdot w_d} + \frac{\deg_E(\tilde D_1 \cdot \ldots \cdot \tilde D_d)}{[\kappa(P) : \Bbbk]},\\
        #2 \deg_E(\tilde D_1 \cdot \ldots \cdot \tilde D_d) & \geq 0,
    \end{align}
    where $\deg_E(\tilde D_1 \cdot \ldots \cdot \tilde D_d)$ denotes the degree of the $0$-cycle class $\tilde D_1 \cdot \ldots \cdot \tilde D_d$ with coefficients in $\mathbb Q$ on a small neighborhood of $E$ and $[\kappa(P) : \Bbbk]$ is the degree of the field extension.
    Moreover, $\deg_E(\tilde D_1 \cdot \ldots \cdot \tilde D_d)$ is zero if and only if $\operatorname{Supp} \tilde D_1 \cap \ldots \cap \operatorname{Supp} \tilde D_d \cap E$ is empty.%
}

\begin{theorem} [=Theorem~\ref{thm:generalization of Fulton for a smooth variety}]
\GeneralizationOfFulton{}{}
\end{theorem}

\subsection{Sextic double solids}

We previously conjectured that factorial sextic double solids with $cA_2$-singularities are birationally superrigid (\cite[Conjecture~1.3]{KOPP24}).
In this paper we consider factorial sextic double solids with ordinary $cA_2$-points (see Definition~\ref{def:ordinary-cA2}) and it turns out that they are birationally rigid but might not be superrigid.
There are two possible divisorial contractions centered on an ordinary $cA_2$ point: regular and exceptional.
The regular one generates an involution in general (Lemma~\ref{lem:nonrigidlink}) and in special cases generates a ``bad link'' (Lemma~\ref{lem:SDSexcldivcont}).
Using Theorem~\ref{MainThm} we show that the exceptional one does not generate a Sarkisov link.
It follows that any birational map $X \dashrightarrow Y$ from a sextic double solid $X$ with $cA_2$-singularities to a Mori fiber space $Y$
is decomposed into involutions of $X$ and a map defined at $cA_2$-singularities (see section~\ref{sc:SDS}).
This results in the following theorem.

\begin{theorem}[$=$ Theorem~\ref{thm:sextic-double-solids}]
    Let $X$ be a factorial sextic double solid with only isolated $cA_1$ points and ordinary $cA_2$ points.
    Then $X$ is birationally rigid.
\end{theorem}

Sextic double solids with non-ordinary $cA_2$-singularities remain out of reach since they admit divisorial contractions with weights $(3,3,2,1)$ and possibly more.
Theorem~\ref{MainThm} is not enough to exclude these contractions and there are infinitely many of them by \cite{Pae24} which makes it unrealistic to consider all potential links originating at $cA_2$-singularities.

\subsection{Rigidity of singular Fano varieties}

According to Noether--Fano inequality, existence of a birational map $X \dashrightarrow Y$ to another Mori fiber space implies existence of $n$ and a mobile linear subsystem $\mathcal M \subset \left| -nK_X \right|$ such that the pair $(X,\frac{1}{n}\mathcal M)$ has worse than canonical singularities at some subvariety $Z \subset X$.
The subvariety $Z$ may be a curve, a non-singular point, or a quotient singularity, or a $cA_k$-point.
Thanks to the previous results: \cite{CPR}, \cite{CP17}, \cite{oka14}, \cite{AZ16}, and \cite{KOPP24} we only need to consider $cA_k$-points.

For each $cA_k$-singular point $P$ we produce a linear system $\mathcal{H}$ such that $P$ is an isolated component of the base locus, we say that the linear equivalence class of $\mathcal{H}$ is a $P$\emph{-isolating class}.
Let $D_1, D_2$ be general divisors in $\mathcal M$ and let $H$ be a general divisor in $\mathcal{H}$, then
the intersection $D_1 \cdot D_2 \cdot H$ provides a natural upper bound on $\operatorname{mult}_P (D_1 \cap D_2)$.
In order to prove birational rigidity we present a $P$-isolating class of sufficiently low degree to get a contradiction with Theorem~\ref{MainThm} for every family in the tables~\ref{table:Fanohyp} and~\ref{table:FanoWCI}.

It is interesting to note that we proved for several number of families of index $1$ Fano complete intersections that $\mathbb Q$-factorial members with $cA_{11}$-points are birationally rigid, and there is even family \textnumero89 such that $\mathbb Q$-factorial members with $cA_{19}$-point is birationally rigid. But we are not sure if there exist $\mathbb Q$-factorial members with such singularities.
It was shown in (\cite[Theorem~A(e)]{PaemurruSDS}) that general sextic double solids with an isolated $cA_k$-point are $\mathbb Q$-factorial for $k=1,\dots, 6, 8$ and the space of sextic double solids with an isolated $cA_7$-point has four connected components, one of which has no $\mathbb Q$-factorial members.
Thus we expect that for most families general member with a $cA_k$-point is $\mathbb Q$-factorial for a small $k$, but we do not know about big $k$.
This motivates the following question.

\begin{question}
For Fano weighted complete intersections, what is the highest $k$ such that the general member of the family with a $cA_k$-point is $\mathbb Q$-factorial?
\end{question}

\subsection{Other works on birational rigidity of singular varieties}

Our inequality works very well for varieties of low degree, but birational rigidity results for singular varieties can be extended to higher degree Fano varieties as well.
For example the birational rigidity of quartic threefolds with one $cA_1$-point is proven in \cite{PukhQuarticA1}.
We expect that with a lot of extra effort similar results can be proven for the families that are not in tables \ref{table:Fanohyp} and~\ref{table:FanoWCI}.

Also note, that in higher dimensions one can obtain stronger inequalities, for example see \cite{Puk17}.
This approach has been applied to prove birational rigidity of many families of high dimensional varieties: hypersurfaces, cyclic covers, complete intersections (\cite{Puk19CI}, \cite{Puk19CO}, \cite{EP18}, \cite{EP19}).

\subsection*{Acknowledgments}
The first author was supported by IBS-R003-D1 grant.
The second author was supported by JAPS KAKENHI Grant Numbers 23K22389 and 24K00519.
The third author was supported by SFB 195 No.~286237555 of DFG, the Simons Foundation, grant SFI-MPS-T-Institutes-00007697, and the Ministry of Education and Science of the Republic of Bulgaria, grant DO1-239/10.12.2024.

\section{Preliminaries} \label{Preliminaries}

An \textit{algebraic $\Bbbk$-scheme}, where $\Bbbk$ is a field, is a scheme of finite type over $\Bbbk$.
A \emph{variety} is an integral separated scheme of finite type over the field of complex numbers~$\mathbb C$.
By the germ of a scheme $X$ at a point~$P$, we mean the direct limit of open subschemes containing~$P$.

\subsection{Weighted blowups}

In Definition~\ref{def:weighted blowup}, we define weighted blowups of arbitrary schemes $X$ (not necessarily of finite type over a field) along a regular system of parameters.
This enables us to state Theorem~\ref{thm:generalization of Fulton for a smooth variety} over an arbitrary field.
Definition~\ref{def:weighted blowup} coincides with the weighted blowup in \cite[Definition~2.2.11]{Kaw24} when $X$ is an algebraic variety (over~$\mathbb C$).
Definition~\ref{def:weighted blowup} coincides with the (usual) blowup along $P$ if all the weights are~$1$.

\begin{definition} \label{def:weighted blowup}
Let $d$ be a positive integer.
Let ${\bm w} = (w_1, \ldots, w_d)$ be a $d$-tuple of positive integers such that $\gcd {\bm w} := \gcd(w_1, \ldots, w_d) = 1$.
Let $X$ be a scheme and $P \in X$ a regular closed point.
We say that a morphism of schemes $\varphi\colon Y \to X$ is the \textit{weighted blowup with weights $w_1, \ldots, w_d$} (or \textit{${\bm w}$-blowup}) with respect to a regular system of parameters $x_1, \ldots, x_d$ at $P$ if $Y$ is isomorphic over $X$ to $\operatorname{Proj}_X R$,
where $R$ is the $\mathbb Z_{\geq0}$-graded $\mathcal O_X$-algebra with stalks
\[
R_Q = \mleft\{\begin{aligned}
    & \mathcal O_{X, P}\mleft[ \mleft\{ t^j x_i \;\middle|\; i \in \{1, \ldots, d\},\, j \in \{0, \ldots, w_i\} \mright\} \mright] && \text{if $Q = P$},\\
    & \mathcal O_{X, Q}[t] && \text{otherwise},
\end{aligned}\mright.
\]
where $t$ denotes the grading.
The \emph{exceptional divisor} of $\varphi$ is defined to be the subscheme of $Y$ given by the surjective $\mathcal O_X$-algebra homomorphism $R \to R / \bigoplus_{m \in \mathbb Z_{\geq0}} R_{m+1} t^m$, where $R_m$ denotes the quasi-coherent $\mathcal O_X$-ideal sheaf that is the $m$-th graded piece of~$R$.
\end{definition}

Note that the exceptional divisor of a weighted blowup need not be Cartier. Next, we define the weight and the least weight part of an element of~$\mathcal O_{X, P}$.

\begin{definition} \label{def:weight}
Given a positive integer $d$ and vector $\bm w = (w_1, \ldots, w_d)$ of positive integers, the weight of an element $g$ of the power series ring $\Bbbk[[x_1, \ldots, x_d]]$ is defined to be
\[
\bm w(g) := \inf \mleft\{ \sum w_i a_i \;\middle|\; \text{the monomial $x_1^{a_1} \cdot \ldots \cdot x_d^{a_d}$ has nonzero coefficient in $g$} \mright\},
\]
where the sum is over $i \in \{1, \ldots, d\}$ and the infimum is over all the monomials in $\Bbbk[x_1, \ldots, x_d]$.
The least weight part of $g \in \Bbbk[[x_1, \ldots, x_d]]$, denoted $g^{\bm w}$, is defined to be the sum of all the monomials of weight~$\bm w(g)$, together with their coefficients in~$g$.

In the setting of Definition~\ref{def:weighted blowup}, the inclusion $\iota\colon \Bbbk [x_1, \ldots, x_d] \to \mathcal O_{X, P}$ induces a $\Bbbk$-algebra isomorphism on the completions $\hat{\iota}\colon \Bbbk [[x_1, \ldots, x_d]] \to \hat{\mathcal O}_{X, P}$.
Define the local monomorphism $\chi \colon \mathcal{O}_{X,P} \to \Bbbk [[x_1,\dots,x_d]]$ to be the composition of the natural injection $\mathcal{O}_{X,P} \to \hat{\mathcal{O}}_{X,P}$ and~$\hat{\iota}^{-1}$.
For every $f \in \mathcal{O}_{X,P}$, we define the weight of $f$ by $\bm w(f) := \bm w(\chi(f))$ and the least weight part of $f$ by $f^{\bm w} := \iota(\chi(f)^{\bm w})$.
\end{definition}

\begin{lemma} \label{lem:weight}
In the setting of Definition~\ref{def:weighted blowup}, for every $f \in \mathcal O_{X, P}$, we have
\[
\bm w(f) = \sup \{ j \in \mathbb Z_{\geq0} \mid t^j f \in R_P \}.
\]
\end{lemma}

\begin{proof}
It suffices to prove the equivalence
\begin{equation} \label{eqn:equivalence}
\bm w(f) \geq k \iff t^k f \in R_P
\end{equation}
for every $f \in \mathcal O_{X, P}$ and positive integer~$k$.
By \cite[Theorem~7.1]{Eis95}, $\chi$ induces a $\Bbbk$-algebra isomorphism
\[
\frac{\mathcal O_{X, P}}{(x_1, \ldots, x_d)^{k+1}} \to \frac{\Bbbk[[x_1, \ldots, x_d]]}{(x_1, \ldots, x_d)^{k+1}}.
\]
So, it suffices to prove equivalence~(\ref{eqn:equivalence}) in the case $f$ is a polynomial of degree at most~$k$, where the statement is obvious.
\end{proof}

Below, if $X$ is affine, then $V_X(I)$ denotes the subscheme of $X$ defined by the ideal $I \subseteq \mathcal O_X(X)$ and $V_Y(J)$ denotes the subscheme of $Y$ defined by the homogeneous ideal $J \in R(X)$.

\begin{lemma} \label{thm:weighted blowup basic description}
Let $d$ be a positive integer.
Let ${\bm w} = (w_1, \ldots, w_d)$ be positive integers such that $\gcd {\bm w} = 1$.
Let $X$ be a scheme and $P$ a regular closed point.
Let $\varphi\colon Y \to X$ be the ${\bm w}$-blowup with respect to a regular system of parameters $x_1, \ldots, x_d$ at $P$.
Then, all of the following hold:
\begin{enumerate}[label=\textup{(\alph*)}, ref=\alph*]
\item \label{itm:weighted blowup basic description - coincides with usual}
if $X$ is an algebraic variety (over the field of complex numbers~$\mathbb C$), then $\varphi$ coincides with the weighted blowup in \cite[Definition~2.2.11]{Kaw24} with $r = 1$,

\item \label{itm:weighted blowup basic description - usual blowup}
if $w_1 = \ldots = w_d = 1$, then $\varphi$ is the (usual) blowup along~$P$,

\item \label{itm:weighted blowup basic description - blowup}
$\varphi$ is the blowup along a quasi-coherent ideal sheaf,

\item \label{itm:weighted blowup basic description - prime divisor}
the exceptional divisor $E$ of $\varphi$ is isomorphic to the scheme $\mathbb P_{\kappa(P)}(w_1, \ldots, w_d) := \operatorname{Proj} \kappa(P)[t^{w_1} x_1, \ldots, t^{w_d} x_d]$ where $t$ denotes the grading and where $\kappa(P)$ is the residue field at~$P$,

\item \label{itm:weighted blowup basic description - exceptional locus}
if $d \geq 2$ and $X$ is Noetherian and integral, then $E$ coincides with the exceptional locus of $\varphi$ with the induced reduced structure,

\item \label{itm:weighted blowup basic description - discrete valuation}
the local ring of $Y$ at $E$ is a discrete valuation ring and the corresponding valuation $v_E$ satisfies $v_E(f) = \bm w(f)$ for all $f \in \mathcal O_{X, P}$,

\item \label{itm:weighted blowup basic description - normal}
if $X$ is normal and integral, then so is~$Y$, and

\item \label{itm:weighted blowup basic description - total transform}
if $X$ is affine, then for all $f \in \mathcal O_X(X) \setminus \{0\}$, the strict transform of $V_X(f)$ is $V_Y(t^{\bm w(f)} f)$.
\end{enumerate}
\end{lemma}

\begin{proof}
We replace $X$ by an affine open neighborhood of $P$ where $x_1, \ldots, x_d$ are defined.
Denote $S := R(X)$ and $S_m := R_m(X)$.
Let $S^{(n)}$ denote the $n$-th Veronese $\mathcal O_X(X)$-subalgebra of~$S$, meaning $S^{(n)} := \bigoplus_{m \in \mathbb Z_{\geq0}} t^m S_{mn}$ where $t$ denotes the degree.

(\ref*{itm:weighted blowup basic description - coincides with usual})
Follows from the text below \cite[Definition~2.2.13]{Kaw24}.

(\ref*{itm:weighted blowup basic description - usual blowup})
Since $S$ is generated in degree~$1$, $\varphi$ coincides with $\operatorname{Proj} \bigoplus_{m \in \mathbb Z_{\geq0}} t^m S_1^m \to X$. The ideal $S_1$ is precisely the prime ideal defining~$P$. Therefore, $\varphi$ is the blowup along~$P$.

(\ref*{itm:weighted blowup basic description - blowup})
For a divisible enough positive integer~$n$, for example $n = d \cdot \operatorname{lcm}(w_1, \ldots, w_d)$, $S^{(n)}$ is generated in degree~$1$ (see the proof of \cite[Proposition~1.6.3]{QR22} for details).
So, $\varphi$ is the blowup along~$S_n$.

(\ref*{itm:weighted blowup basic description - prime divisor})
Follows from the fact that $E$ is given by $V_Y(\bigoplus_{m \in \mathbb Z_{\geq0}} t^m S_{m+1})$.

(\ref*{itm:weighted blowup basic description - exceptional locus})
By~(\ref{itm:weighted blowup basic description - blowup}), $\varphi$ is the blowup along~$S_n$.
For all positive integers $k < l$ and elements $f \in S_l$, we have $f^{l-k} \cdot (t^l f)^k = (t^k f)^l$.
Therefore, the radical of $S_n$ is $\bigoplus_{m \in \mathbb Z_{\geq0}} t^m S_{m+1}$.
So, the support of the exceptional Cartier divisor of the blowup along $S_n$ coincides with~$E$.
Therefore, $\varphi$ restricts to an isomorphism $Y \setminus E \to X \setminus \{P\}$.
By \cite[Corollary~13.97]{GW20}, $Y$ is integral and $\varphi$ is proper and birational.
Since $d \geq 2$, we can use \cite[Lemma~2.20]{Liu06} to conclude that the exceptional locus of $\varphi$ is~$E$.

(\ref*{itm:weighted blowup basic description - discrete valuation})
Taking the affine open $U := D_+(t^{w_1} x_1)$ of~$Y$, let $I \subseteq \mathcal O_Y(U)$ be the ideal defining $E \cap U$. Then, $\mathcal O_Y(U)$ consists of the elements $g / x_1^k$ where the nonnegative integer $k$ and element $g \in \mathcal O_X(X)$ are such that there exists a nonnegative integer $m$ satisfying $g \in S_m$ and $m \geq k w_1$.
By Lemma~\ref{lem:weight}, the coordinate ring $\mathcal O_Y(U)$, respectively the ideal~$I$, consists of the elements $g / x_1^k$ such that $\bm w(g) \geq k w_1$, respectively $\bm w(g) > k w_1$.
We see that $I$ is a prime ideal and the localization at the complement of $I$ is a discrete valuation ring. Since $\gcd(w_1, \ldots, w_d) = 1$, the claim follows.

(\ref*{itm:weighted blowup basic description - normal})
We show that $Y$ is normal.
By~(\ref{itm:weighted blowup basic description - discrete valuation}), for every nonnegative integer~$m$, the ideal $S_m$ consists precisely of the elements $f \in \mathcal O_X(X)$ such that $v_E(f) \geq m$.
It follows that $S_m$ is a normal ideal in $\mathcal O_X(X)$ for every nonnegative integer~$m$, meaning that all the positive powers of $S_m$ are integrally closed in $\mathcal O_X(X)$.
By (\ref{itm:weighted blowup basic description - blowup}), $\varphi$ is the blowup along~$S_n$.
Since $\mathcal O_X(X)$ is a normal domain, we find by \cite[Proposition~5.2.4]{HS06} that $\mathcal O_X(X)[t S_n]$ is a normal domain.
Now, suffices to use the result that if $A$ is a $\mathbb Z_{\geq0}$-graded normal domain, then $\operatorname{Proj} A$ is normal.

(\ref*{itm:weighted blowup basic description - total transform})
By \cite[Corollary~3.2.14]{QR22}, the strict transform $\overline{\varphi^{-1}(V_X(f)) \setminus E}$ of $V_X(f)$ is given by the surjective $\mathbb Z_{\geq0}$-graded $\mathcal O_X(X)$-algebra homomorphism
$S \to \operatorname{Proj}\mleft(\frac{\mathcal O_X(X)}{(f)}\mleft[ \mleft\{ t^j \bar{x}_i \mright\}_{i,j} \mright]\mright)$,
where $\bar{x}_i$ denotes the equivalence class $x_i + (f)$ and where we use the notation
\[
\mleft\{ t^j \bar{x}_i \mright\}_{i,j} := \mleft\{ t^j \bar{x}_i \mid i \in \{1, \ldots, d\},\, j \in \{0, \ldots, w_i\} \mright\}.
\]
Considering each homogeneous part separately, we find
\[
\frac{\mathcal O_X(X)}{(f)}\mleft[ \mleft\{ t^j \bar{x}_i \mright\}_{i,j} \mright]
= \frac{S}{(f, tf, \ldots, t^{\bm w(f)}f)}.
\]
Since the intersections $(f, tf, \ldots, t^{\bm w(f)}f) \cap S_{\geq\bm w(f)}$ and $(t^{\bm w(f)} f) \cap S_{\geq\bm w(f)}$ coincide, the ideals $(f, tf, \ldots, t^{\bm w(f)}f)$ and $(t^{\bm w(f)} f)$ define the same subscheme of~$Y$. The claim follows.
\end{proof}

\subsection{Intersection theory}

The degree of a $\mathbb Z$-cycle class is defined in \cite[Definition~1.4]{Ful98}, the intersection of Cartier divisors is defined in \cite[Definition~2.4.2]{Ful98} and the intersection multiplicity of Cartier divisors at a point is defined in \cite[Example~7.1.10(a)]{Ful98}.
In Definition~\ref{def:degree of Q-cycles} and Remark~\ref{def:intersection of Q-Cartier divisors}, we define the degree of a $\mathbb Q$-cycle class, the intersection of $\mathbb Q$-Cartier divisors and the intersection multiplicity of $\mathbb Q$-Cartier divisors at a point.
In this section, the ground field $\Bbbk$ is an arbitrary field.

\begin{definition} \label{def:degree of Q-cycles}
Let $X$ be an algebraic $\Bbbk$-scheme and let $i$ be a positive integer.
We denote by $Z_i (X)$ (resp.\ $A_i (X)$) the group of $i$-cycles on $X$ (resp.\ the group of $i$-cycles modulo rational equivalence on $X$).
We define
\[
\begin{split}
Z_i (X)_{\mathbb{Q}} &:= Z_i (X) \otimes_{\mathbb{Z}} \mathbb{Q}, \\
A_i (X)_{\mathbb{Q}} &:= A_i (X) \otimes_{\mathbb{Z}} \mathbb{Q}.
\end{split}
\]
If $X$ is proper, then the degree map $\deg \colon A_0 (X) \to \mathbb{Z}$ for $0$-cycles is extended to
\[
\deg \colon A_0 (X)_{\mathbb{Q}} \to \mathbb{Q}.
\]
\end{definition}

\begin{definition} \label{def:multiplicity}
Let $X$ be an algebraic $\Bbbk$-scheme and let $S \subset X$ be an integral subscheme.
The \textit{multiplicity of $X$ along $S$}, denoted by $\operatorname{mult}_S X$, is defined to be the Hilbert--Samuel multiplicity $e (\mathfrak{m}_{X,S}, \mathcal{O}_{X, S})$ of the local ring $\mathcal{O}_{X, S}$ at its maximal ideal $\mathfrak{m}_{X,S}$.
Let $Y$ be an algebraic $\Bbbk$-scheme, let $\alpha = \sum a_i Z_i$ be a rational cycle on~$Y$, where $Z_i$ is an integral subscheme for any $i$ and $a_i \in \mathbb{Q}$, and let $S \subset Y$ be an integral subscheme of $Y$.
Then the \textit{multiplicity of $\alpha$ along $S$} is defined by
\[
\operatorname{mult}_S (\alpha) := \sum a_i \operatorname{mult}_S (Z_i),
\]
where we set $\operatorname{mult}_S (Z_i) = 0$ if $S$ is not contained in $Z_i$.
\end{definition}

\begin{remark} \label{def:intersection of Q-Cartier divisors}
Let $X$ be a $d$-dimensional normal integral algebraic $\Bbbk$-scheme.
Let $D_1, \dots, D_s$ be $\mathbb{Q}$-Cartier Weil divisors on~$X$, where $2 \le s \le d$.
Let $b_i$ be positive integers such that $b_i D_i$ is a Cartier divisor.
Then there is a well-defined $(d-s)$-cycle class
\[
(b_1 D_1) \cdot \ldots \cdot (b_s D_s) \in A_{d-s} \bigl(\operatorname{Supp} (D_1) \cap \ldots \cap \operatorname{Supp} (D_s)\bigr).
\]
We define
\[
D_1 \cdot \ldots \cdot D_s := \frac{(b_1 D_1) \cdot \ldots \cdot (b_d D_d)}{b_1 \cdots b_d} \in A_{d-s} \bigl(\operatorname{Supp} (D_1) \cap \ldots \cap \operatorname{Supp} (D_s)\bigr)_{\mathbb{Q}}.
\]

Suppose in addition that $D_1, \dots, D_s$ are effective.
If $Z$ is a $(d-s)$-dimensional irreducible component of $\operatorname{Supp} (D_1) \cap \cdots \cap \operatorname{Supp} (D_s)$, then the coefficient of $[Z]$ in $D_1 \cdot \ldots \cdot D_s$ is called the \textit{intersection multiplicity of $D_1, \dots, D_s$ at $Z$} and it coincides with the multiplicity $\operatorname{mult}_Z (D_1 \cdot \ldots \cdot D_s)$ defined in Definition~\ref{def:multiplicity} (see \cite[Example 7.1.10]{Ful98}).
\end{remark}

We have the following well-known lemma.
Below, a general hyperplane on a germ of a regular closed point $P \in X$ is a Cartier divisor defined at $P$ by an equation $f \in \mathfrak{m}_{X,P}$ such that its image in $\mathfrak{m}_{X,P}/\mathfrak{m}_{X,P}^2$ is general.

\begin{lemma} \label{thm:intersect with general hyperplanes}
Let $P \in X$ be a germ of a regular closed point of an algebraic $\Bbbk$-scheme, let $Z \in Z_j (X)_{\mathbb{Q}}$ be an effective rational $j$-cycle on $X$ and let $H_1, \dots, H_i$ be general hyperplanes through~$P$, where $i \le j$ are nonnegative integers.
Then,
\[
\operatorname{mult}_P (Z) = \operatorname{mult}_P (Z \cdot H_1 \cdot \ldots \cdot H_i).
\]
\end{lemma}

\begin{proof}
If suffices to prove the case where $P$ is in the support of~$Z$.
Let $H_1, \dots, H_i$ be general hyperplanes through $P$ so that $\operatorname{Supp} (Z) \cap H_1 \cap \cdots \cap H_i$ has pure dimension $j-i$ after possibly shrinking $P \in X$.
We take general hyperplanes $H_{i+1}, \dots, H_j$ so that $\operatorname{Supp} (Z) \cap H_1 \cap \cdots \cap H_j$ consists only of the point $P$ after possibly shrinking $P \in X$.
Since $H_1, \ldots, H_j$ are general, their strict transforms have empty intersection with the strict transform of $\operatorname{Supp} (Z)$ under the blowup at~$P$.
By \cite[Example~12.4.8]{Ful98} (or Theorem~\ref{thm:generalization of Fulton for a smooth variety} for the usual blowup), we have
\[
\operatorname{mult}_P (Z) = \operatorname{mult}_P (Z \cdot H_1 \cdot \ldots \cdot H_j).
\]
Applying this result for the effective $(j - i)$-cycle $Z \cdot H_1 \cdot \ldots \cdot H_i$, we obtain
\[
\operatorname{mult}_P (Z \cdot H_1 \cdot \ldots \cdot H_i) = \operatorname{mult}_P (Z \cdot H_1 \cdot \ldots \cdot H_j),
\]
and the lemma follows.
\end{proof}

In Theorem~\ref{thm:locineqwblcA}, we use local analytic coordinate changes. The lemma below guarantees that this does not alter intersection numbers.
In the lemma, we use $\hat{\mathcal O}_{X, P}$ to denote the completion of the local ring $\mathcal O_{X, P}$ and we use $f^{\leq k}$ to denote the truncation of a power series $f \in \Bbbk[[x_1, \ldots, x_d]]$ up to degree~$k$.

\begin{lemma} \label{lem:truncate}
Let $d$ be a positive integer.
Let $P \in X$ be a $\Bbbk$-rational point of a smooth integral $d$-dimensional algebraic $\Bbbk$-scheme.
Let $\Psi\colon \hat{\mathcal O}_{X, P} \to \Bbbk[[x_1, \ldots, x_d]]$ be a $\Bbbk$-algebra isomorphism.
Let $D_1, \ldots, D_d$ be effective Cartier divisors on $X$ such that $P$ is an isolated point of $\operatorname{Supp} D_1 \cap \ldots \cap \operatorname{Supp} D_d$.
Let $f_1, \ldots, f_d \in \Bbbk[[x_1, \ldots, x_d]]$ be the images under $\Psi$ of chosen local equations of respectively $D_1, \ldots, D_d$.
Then, for every sufficiently large positive integer~$N$, the origin $\bm 0$ is an isolated point of the intersection of the supports of the divisors $V(f_1^{\leq N}), \ldots, V(f_d^{\leq N})$ on $\mathbb A_{\Bbbk}^d$ and
\[
\operatorname{mult}_P(D_1 \cdot \ldots \cdot D_d) = \operatorname{mult}_{\bm 0}\mleft(V(f_1^{\leq N}) \cdot \ldots \cdot V(f_d^{\leq N})\mright).
\]
\end{lemma}

\begin{proof}
Let $F_1, \ldots, F_d$ denote the chosen local equations of respectively $D_1, \ldots, D_d$.
Since $X$ is regular, by \cite[Example~7.1.10(a)]{Ful98}, the intersection multiplicity is computed by the length of a $\mathcal O_{X, P}$-module as below,
\[
\operatorname{mult}_P(D_1 \cdot \ldots \cdot D_d) = l_{\mathcal O_{X, P}}(\mathcal O_{X, P} / (F_1, \ldots, F_d)).
\]
As written in \cite[Example~7.1.10(b)]{Ful98}, modules of finite length are not altered by completion. We find
\[
\operatorname{mult}_P(D_1 \cdot \ldots \cdot D_d) = l_{\Bbbk[[x_1, \ldots, x_d]]}\mleft(\frac{\Bbbk[[x_1, \ldots, x_d]]}{(f_1, \ldots, f_d)}\mright).
\]
By \cite[Exercise~7.1.6(d)]{Liu06}, we have
\[
\operatorname{mult}_P(D_1 \cdot \ldots \cdot D_d) = \dim_{\Bbbk} \frac{\Bbbk[[x_1, \ldots, x_d]]}{(f_1, \ldots, f_d)}.
\]
Let $N_0$ be a positive integer such that $(x_1, \ldots, x_d)^{N_0} \subseteq (f_1, \ldots, f_d)$. Then, for every positive integer $N \geq N_0$, the origin $\bm 0$ is an isolated point of the intersection of the divisors $V(f_1^{\leq N}), \ldots, V(f_d^{\leq N})$ on~$\mathbb A_{\Bbbk}^d$ and
\[
\dim_{\Bbbk} \frac{\Bbbk[[x_1, \ldots, x_d]]}{(f_1, \ldots, f_d)} = \dim_{\Bbbk} \frac{\Bbbk[[x_1, \ldots, x_d]]}{(f_1^{\leq N}, \ldots, f_d^{\leq N})}.
\]
By \cite[Example~7.1.10(b)]{Ful98},
\[
\dim_{\Bbbk} \frac{\Bbbk[[x_1, \ldots, x_d]]}{(f_1^{\leq N}, \ldots, f_d^{\leq N})} = \operatorname{mult}_{\bm 0}\mleft(V(f_1^{\leq N}) \cdot \ldots \cdot V(f_d^{\leq N})\mright). \qedhere
\]
\end{proof}

\begin{definition}
Let $X$ be an algebraic $\Bbbk$-scheme and let $P \in X$ be a closed point.
For an effective Cartier divisor $D$ on $X$, we define
\[
\operatorname{ord}_P (D) := \sup \{\, j \ge 0 \mid f \in \mathfrak{m}_{X, P}^j \,\},
\]
where $f \in \mathcal{O}_{X, P}$ is the local equation of $D$.
In general, for a $\mathbb{Q}$-Cartier $\mathbb{Q}$-divisor $D = \sum a_i D_i$, where $a_i \in \mathbb{Q}$ and $D_i$ are effective Cartier divisors, we define
\[
\operatorname{ord}_P (D) := \sum a_i \operatorname{ord}_P (D_i),
\]
and call it the \textit{order} of $D$ at $P$.
\end{definition}

Note that $\operatorname{ord}_P (D) = \operatorname{mult}_P (D)$ if $X$ is regular at $P$.

The following result is useful when applying Theorem~\ref{body-MainThm}.%

\begin{lemma} \label{lem:intom}
Let $X$ be a normal $d$-dimensional integral scheme that is proper over an arbitrary field $\Bbbk$ with $d \ge 3$.
Let $D_1, D_2, S_1, \dots, S_{d-3}$ be effective $\mathbb{Q}$-Cartier $\mathbb{Q}$-divisors on $X$ that intersect properly, that is, any irreducible component of
\[
\Gamma := \operatorname{Supp} (S_1) \cap \cdots \cap \operatorname{Supp} (S_{d-3}) \cap \operatorname{Supp} (D_1) \cap \operatorname{Supp} (D_2)
\]
is a curve, and let $P$ be a $\Bbbk$-point of $X$ such that $P \in X$ is a local complete intersection singularity.
If $T$ is an effective $\mathbb{Q}$-Cartier $\mathbb{Q}$-divisor on $X$ such that it is nef and $\operatorname{Supp} (T)$ does not contain any irreducible component of $\Gamma$ passing through $P$, then
\[
T \cdot S_1 \cdot \ldots \cdot S_{d-3} \cdot D_1 \cdot D_2 \ge \mathrm{ord}_P (T) \operatorname{ord}_P (S_1) \cdots \operatorname{ord}_P (S_{d-3}) \operatorname{mult}_P (D_1 \cdot D_2).
\]
\end{lemma}

\begin{proof}
Replacing $D_1, D_2, S_1, \dots, S_{d-3}$ and $T$ by their positive multiples, we may assume that they are Cartier divisors.
We decompose the effective $1$-cycle
\[
S_1 \cdot \ldots \cdot S_{d-3} \cdot D_1 \cdot D_2 = C + C',
\]
where $C$ and $C'$ are effective $1$-cycles on $X$ such that any curve in $\operatorname{Supp} (C)$ passes through $P$ and no curve in $\operatorname{Supp} (C')$ passes through~$P$.
We then obtain
\begin{equation} \label{eq:intom1}
\begin{split}
T \cdot S_1 \cdot \ldots \cdot S_{d-3} \cdot D_1 \cdot D_2 &\ge (T \cdot C) \\
&\ge \operatorname{mult}_P(T \cdot C) \\
&= \operatorname{mult}_P(T \cdot S_1 \cdot \ldots \cdot S_{d-3} \cdot D_1 \cdot D_2).
\end{split}
\end{equation}
where the first inequality holds since $T$ is nef and the second inequality holds since $T \cap C$ is a finite set of points.

From here on, we work locally around $P$.
In a neighborhood of $P$, $X$ can be identified with the complete intersection $\mathcal{X}_1 \cap \cdots \cap \mathcal{X}_c$ in a smooth $(d+c)$-dimensional integral scheme $\mathcal{X}$ of finite type over~$\Bbbk$, where $\mathcal{X}_i$ are hypersurfaces in $\mathcal{X}$.
Note that $P$ is necessarily a smooth point of $\mathcal{X}$ since it is a regular $\Bbbk$-point of~$\mathcal{X}$.
We can take Cartier divisors $\mathcal{D}_1, \mathcal{D}_2, \mathcal{S}_1, \dots, \mathcal{S}_{d-3}$ and $\mathcal{T}$ on $\mathcal{X}$ such that $\mathcal{D}_i|_X = D_i$ for $i = 1, 2$, $\mathcal{S}_j|_X = S_j$ for $i = 1, \dots, d-3$ and $\mathcal{T}|_X = T$.
We may assume that $\operatorname{ord}_P (S_j) = \operatorname{ord}_P (\mathcal{S}_j) = \operatorname{mult}_P (\mathcal{S}_j)$ for $j = 1, \dots, d-3$ and $\operatorname{ord}_P (T) = \operatorname{ord}_P (\mathcal{T}) = \operatorname{mult}_P (\mathcal{T})$.
Then we have
\begin{equation} \label{eq:intom2}
\begin{split}
& \operatorname{mult}_P(T \cdot S_1 \cdot \ldots \cdot S_{d-3} \cdot D_1 \cdot D_2) \\
&= \operatorname{mult}_P(\mathcal{T} \cdot \mathcal{S}_1 \cdot \ldots \cdot \mathcal{S}_{d-3} \cdot \mathcal{D}_1 \cdot \mathcal{D}_2 \cdot \mathcal{X}_1 \cdot \ldots \cdot \mathcal{X}_c) \\
&\ge \operatorname{mult}_P (\mathcal{T}) \operatorname{mult}_P (\mathcal{S}_1) \cdots \operatorname{mult}_P (\mathcal{S}_{d-3}) \operatorname{mult}_P (\mathcal{D}_1 \cdot \mathcal{D}_2 \cdot \mathcal{X}_1 \cdot \ldots \cdot \mathcal{X}_c) \\
&= \operatorname{ord}_P (T) \operatorname{ord}_P (S_1) \cdots \operatorname{ord}_P (S_{d-3}) \operatorname{mult}_P (D_1 \cdot D_2),
\end{split}
\end{equation}
where the inequality follows from \cite[Corollary 12.4]{Ful98} since $P \in \mathcal{X}$ is a smooth point.
The assertion follows from \eqref{eq:intom1} and \eqref{eq:intom2}.
\end{proof}

\subsection{The method of maximal singularities for Fano 3-folds}

In this section, the ground field is assumed to be $\mathbb{C}$.
We recall definitions concerning maximal singularities and explain the characterization of birational rigidity of Fano $3$-folds in terms of maximal singularities.
Throughout the paper, a \textit{Fano variety} is a normal projective $\mathbb{Q}$-factorial variety with only terminal singularities whose anticanonical divisor is ample.
By a \textit{divisorial contraction} $\varphi \colon Y \to X$, we mean a proper birational morphism between varieties with only terminal singularities such that $-K_Y$ is $\varphi$-ample and that the exceptional locus $E$ is a prime divisor on $Y$.

Let $X$ be a Fano $3$-fold of Picard number $1$.

\begin{definition}
A divisorial contraction $\varphi \colon Y \to X$ is called a \textit{maximal extraction} if there is a mobile linear system $\mathcal{M} \sim_{\mathbb{Q}} - n K_X$, where $n > 0$ is a rational number, such that the pair $(X, \frac{1}{n} \mathcal{M})$ is not canonical along $E$, that is, the inequality $\operatorname{mult}_E (\mathcal{M}) > n a_E(X)$ holds, where $a_E (X) := \operatorname{ord}_E K_{Y/X}$ is the discrepancy of $X$ along $E$.
The center of a maximal extraction is called a \textit{maximal center}.
\end{definition}

\begin{definition}
A birational automorphism $\sigma \colon X \dashrightarrow X$ is called an \textit{elementary self-link} or a \textit{Sarkisov self-link} if there is a divisorial contraction $\varphi \colon Y \to X$, a birational map $\tau \colon Y \dashrightarrow Y'$ which is a composite of inverse flips, a flop and flips between normal projective $\mathbb{Q}$-factorial $3$-folds with only terminal singularities and a divisorial contraction $\varphi' \colon Y' \to X$ such that the diagram
\[
\xymatrix{
Y \ar[d]_{\varphi} \ar@{-->}[r]^{\tau} & Y' \ar[d]^{\varphi'} \\
X \ar@{-->}[r]_{\sigma} & X}
\]
commutes.
In this case, we say that the elementary self-link $\sigma$ is \textit{initiated} by $\varphi$.
\end{definition}

The following is a characterization of birational rigidity which is well-known to experts.
We refer the reader to \cite[Lemma 2.34]{Oka18} for the proof (of a more general result).

\begin{theorem} \label{thm:chractBR}
Let $X$ be a Fano $3$-fold of Picard number $1$.
Then $X$ is birationally rigid if and only if, for each divisorial contraction $\varphi \colon Y \to X$,
\begin{itemize}
\item either $\varphi$ is not a maximal extraction,
\item or there exists an elementary self-link initiated by $\varphi$.
\end{itemize}
\end{theorem}

\subsection{\texorpdfstring
    {Divisorial contractions to $cA_k$-points}
    {Divisorial contractions to cA\_k-points}
}

In this section, the ground field is assumed to be $\mathbb{C}$.
We recall the classification of divisorial contractions with center an isolated $cA_k$ point.

\begin{definition}
Let $k$ be a positive integer.
A $3$-fold germ $P \in X$ is called a $cA_k$-singularity (or a $cA_k$-point) if its general hyperplane section $P \in S$ is a Du Val singularity of type $A_k$.
\end{definition}

It is known that a $cA_k$ singularity is terminal if and only if it is isolated.

\begin{theorem}[{\cite[Theorem 1.13]{Kaw03}, \cite[Theorem 1.1]{Kaw02}, \cite[Theorem 2.6]{Yam18} and \cite[Theorem 6.10]{Pae24}}]
\label{thm:cldccA2}
Let $P \in X$ be an isolated $cA_k$-point, where $k \ge 1$, and let $\varphi$ be a divisorial contraction with center $P \in X$.
Then one of the following holds.
\begin{enumerate}
    \item After locally analytically identifying $P \in X$ with the germ
    \[
    o \in (x_1 x_2 + g (x_3, x_4) = 0) \subset \mathbb{A}^4,
    \]
    $\varphi$ is the weighted blow-up with weights $\operatorname{wt} (x_1, x_2, x_3, x_4) = (r_1, r_2, a, 1)$, where $r_1 + r_2 =  (k+1)a$ and $a$ is co-prime to $r_1$ and $r_2$, $g$ has weighted order $r_1 + r_2$ with weights $\operatorname{wt} (x_1, x_4) = (a, 1)$, and the monomial $x_3^{(r_1+r_2)/a} \in g$.
    \item The singularity $P \in X$ is of type $cA_1$ and, after locally analytically identifying $P \in X$ with the germ
    \[
    o \in (x_1 x_2 + x_3^2 + x_4^3 = 0) \subset \mathbb{A}^4,
    \]
    $\varphi$ is the weighted blow-up with weights $\operatorname{wt} (x_1, x_2, x_3, x_4) = (1, 5, 3, 2)$.
    \item The singularity $P \in X$ is of type $cA_2$ and, after locally analytically identifying $P \in X$ with the germ
    \[
    o \in (x_1^2 + x_2^2 + x_3^3 + x_1 x_4^2 = 0) \subset \mathbb{A}^4,
    \]
    $\varphi$ is the weighted blow-up with weights $\operatorname{wt} (x_1, x_2, x_3, x_4) = (4, 3, 2, 1)$.
\end{enumerate}
\end{theorem}

\begin{remark}
    Let $\varphi$ be a divisorial contraction with center a $cA_k$ point $P \in X$.
    If $\varphi$ is the one described in (1) (resp.\ (2), resp.\ (3)) of Theorem~\ref{thm:cldccA2}, then its discrepancy is $a$ (resp.\ $4$, resp.\ $3$).
\end{remark}

We introduce a subclass of $cA_k$-singularity, called an \textit{ordinary} $cA_k$-\textit{singularity} that will be necessary in section~\ref{sc:SDS}.

\begin{definition}
    Let $P \in X$ be an isolated $cA_k$-point and $\varphi \colon Y \to X$ be a divisorial contraction with center $P$.
    The \textit{weights} of $\varphi$ is defined to be $(r_1, r_2, a, 1)$ (resp.\ $(1, 5, 3, 2)$, resp.\ $(4, 3, 2, 1)$) if $\varphi$ corresponds to (1) (resp.\ (2), resp.\ (3)) of Theorem~\ref{thm:cldccA2}.
    A divisorial contraction $\varphi$ is called \textit{exceptional} if it corresponds to either (2) or (3) of Theorem~\ref{thm:cldccA2}.
\end{definition}

\begin{definition} \label{def:ordinary-cA2}
    Let $P \in X$ be an isolated $cA_k$ point.
    We say that $P \in X$ is \textit{ordinary} if it does not admit a divisorial contraction of discrepancy greater than $1$ that is not exceptional.
\end{definition}

Divisorial contractions to an isolated ordinary $cA_k$ point is either a divisorial contraction described in (1) of Theorem~\ref{thm:cldccA2} with $a = 1$ or an exceptional divisorial contraction.
Most of the isolated $cA_2$ points are ordinary as the following remarks suggest.

\begin{remark}
    Let $P \in X$ be an isolated $cA_2$ point defined in $\mathbb{C}^4$ by the equation
    \[
    x y + f (z, t) = 0,
    \]
    where $f$ has order $3$.
    If the degree $3$ homogeneous part of $f$ is not the cube of a linear form in $z$ and $t$, then $P \in X$ is ordinary.
\end{remark}

\begin{remark}
    Let $P \in X$ be the isolated $cA_2$ point defined in $\mathbb{A}^4$ by the equation
    \[
    x y + z^3 + t^m = 0,
    \]
    where $m \ge 3$.
    Then it is ordinary if and only if $m \le 5$.
    If $m = 3$, then $P \in X$ does not admit a divisorial contraction of exceptional type, so that it admits 2 divisorial contractions of discrepancy $1$.
    If $m = 4$, then $P \in X$ is equivalent to
    \[
    x^2 + y^2 + z^3 + x t^2 = 0
    \]
    and, with this identification, the weighted blowup with weights $\operatorname{wt} (x, y, z, t) = (4, 3, 2, 1)$ is a divisorial contraction of exceptional type.
\end{remark}

\section{Local inequalities}

\subsection{Main inequality}

Theorem~\ref{thm:generalization of Fulton for a smooth variety} is a generalization of \cite[Theorem~3]{FL82} and \cite[Example~12.4.8]{Ful98} to weighted blowups with respect to a regular system of parameters.

\begin{theorem} \label{thm:generalization of Fulton for a smooth variety}
\GeneralizationOfFulton{\label{eqn:weighted mult equality}}{\label{eqn:weighted mult inequality}}
\end{theorem}

\begin{proof}
We follow the proof of \cite[Theorem~3]{FL82}.
By replacing $D_i$ with their positive multiples, we may assume that $D_i$ are Cartier divisors.
Let $b$ be the least common multiple of $w_1, \ldots, w_d$.
Let $x_1, \ldots, x_d$ be the regular system of parameters at $P$ along which the weighted blowup was taken.
By \cite[Example~2.4.2]{Ful98}, we can replace $X$ with a suitable open neighborhood of~$P$ such that $D_1 \cap \ldots \cap D_d = \{P\}$ and such that $x_1, \ldots, x_d$ are defined on~$X$.
For every Weil divisor $D$ on~$Y$, $bD$ is Cartier.
Let $\eta\colon E \to P$ be the restriction of~$\varphi$.
Using the projection formula \cite[Proposition~2.3]{Ful98},
both of the following hold:
\begin{enumerate}[label=(\roman*), ref=\roman*]
\item $\eta_*(\varphi^* D_1 \cdot \ldots \cdot \varphi^* D_d) = D_1 \cdot \ldots \cdot D_d$,
\item for all $j \in \{1, \ldots, d-1\}$, $\eta_*(\varphi^* D_1 \cdot \ldots \cdot \varphi^* D_j \cdot E^{d-j}) = 0$.
\end{enumerate}
Denote $m_i := v_E(D_i)$. Equivalently,
\begin{enumerate}[label=(\roman*), ref=\roman*, resume]
\item \label{itm:multiplicity} for all $i \in \{1, \ldots, d\}$, $\varphi^*(D_i) = \tilde{D_i} + m_i E$.
\end{enumerate}
Let $H_i$ be the Weil divisor $V(x_i)$ on $X$ and let $\tilde{H}_i$ denote the strict transform of the divisor $H_i$ on~$Y$.
For all $i$, we have
\[
\varphi^*(H_i) = \tilde{H}_i + w_i E.
\]
By (i)--(iii) and bilinearity of intersection products,
\[
\begin{aligned}
    0 & = \eta_*\bigl(b \tilde{H}_1 \cdot \ldots \cdot b \tilde{H}_d \bigr)\\
    & = \eta_*\Bigl( b \bigl(\varphi^*(H_1) - w_1 E\bigr) \cdot \ldots \cdot b \bigl(\varphi^*(H_d) - w_d E\bigr)\Bigr)\\*
    & = b^d (H_1 \cdot \ldots \cdot H_d) + (-1)^d \cdot w_1 \cdot \ldots \cdot w_d \cdot \eta_*\bigl((bE)^d\bigr).
\end{aligned}
\]
Therefore,
\[
\eta_*\bigl((bE)^d\bigr) = \frac{(-1)^{d-1} b^d}{w_1 \cdot \ldots \cdot w_d} (H_1 \cdot \ldots \cdot H_d).
\]
After again replacing $X$ with a suitable open neighborhood of~$P$, we have $H_1 \cdot \ldots \cdot H_d = [P]$ and we find
\begin{enumerate}[label=(\roman*), ref=\roman*, resume]
\item $\eta_*\bigl((bE)^d\bigr) = \dfrac{(-1)^{d-1} b^d}{w_1 \cdot \ldots \cdot w_d}[P]$.
\end{enumerate}
By (i)--(iv) and bilinearity of intersection products,
\[
\begin{aligned}
    \eta_*(b \tilde D_1 \cdot \ldots \cdot b \tilde D_d)
    & = \eta_*\Bigl( b\bigl(\varphi^*(D_1) - m_1 E\bigr) \cdot \ldots \cdot b\bigl(\varphi^*(D_d) - m_d E\bigr) \Bigr)\\*
    & = b^d \mleft(D_1 \cdot \ldots \cdot D_d - \frac{m_1 \cdot \ldots \cdot m_d}{w_1 \cdot \ldots \cdot w_d} [P]\mright).
\end{aligned}
\]
The intersection class $b \tilde D_1 \cdot \ldots \cdot b \tilde D_d$ is a zero-cycle up to rational equivalence on $\tilde D_1 \cap \ldots \cap \tilde D_d \subseteq E$. By \cite[Theorem~1.4 and Definition~1.4]{Ful98}, all rationally equivalent zero-cycles $\alpha$ on $E$ have the same degree, which coincides with the degree of $\eta_*(\alpha)$.
This proves equation~\ref{eqn:weighted mult equality}.

Now we prove inequality~\ref{eqn:weighted mult inequality}.
Let $L$ be the ample line bundle $\mleft.\mathcal O_Y(-bE)\mright|_E$.
We shrink $X$ if necessary so that the divisors $D_1, \ldots, D_d$ are principal.
By (\ref{itm:multiplicity}), $\mathcal O(\tilde D_i) = \mathcal O(-m_i E)$.
Since every $b \tilde D_i$ is an effective Cartier divisor of~$Y$, the closed embedding
\[
\iota\colon b \tilde D_1 \times \ldots \times b \tilde D_d \to Y \times \ldots \times Y
\]
is a regular embedding of codimension~$d$. By \cite[Appendix B.6.2 or B.7.1]{Ful98} or \cite[Corollary~6.3.8]{Liu06}, the normal sheaf to $b \tilde D_1 \times \ldots \times b \tilde D_d$ in $Y \times \ldots \times Y$ is a vector bundle $\mathcal N$ of rank~$d$ on $b \tilde D_1 \times \ldots \times b \tilde D_d$.
Consider the commutative diagram
\[
\begin{tikzcd}
b \tilde D_1 \cap \ldots \cap b \tilde D_d \arrow[r, "", hook] \arrow[d, "g", hook] & Y \arrow[d, "", hook]\\
b \tilde D_1 \times \ldots \times b \tilde D_d \arrow[r, "\iota", hook] & Y \times \ldots \times Y
\end{tikzcd}
\]
where $Y \to Y \times \ldots \times Y$ is the diagonal embedding and $b \tilde D_1 \cap \ldots \cap b \tilde D_d$ is the scheme-theoretic intersection.
By \cite[Example~12.4.8]{Ful98}, the intersection class $b\tilde D_1 \cdot \ldots \cdot b\tilde D_d$ coincides with $s^*[C_{b \tilde D_1 \cap \ldots \cap b \tilde D_d} Y]$, where
\[
s\colon b \tilde D_1 \cap \ldots \cap b \tilde D_d \to g^* \mathcal N
\]
is the zero-section and $C_{b \tilde D_1 \cap \ldots \cap b \tilde D_d} Y$ is the normal cone to $b\tilde D_1 \cap \ldots \cap b\tilde D_d$ in~$Y$.
We have
\[
g^* \mathcal N = \bigoplus_{i \in \{1, \ldots, d\}} \mleft.L\mright|_{b \tilde D_1 \cap \ldots \cap b \tilde D_d}^{\otimes m_i}.
\]
Note that $L$ is a globally generated line bundle.
Therefore, $L^{\otimes m_i} \otimes L^{\vee}$ is generated by global sections. Since a finite direct sum of globally generated sheaves is globally generated,
we find that
\[
\mleft( \bigoplus_{i \in \{1, \ldots, d\}} L^{\otimes m_i} \mright) \otimes L^{\vee}
= \bigoplus_{i \in \{1, \ldots, d\}} \mleft( L^{\otimes m_i} \otimes L^{\vee} \mright)
\]
is globally generated.
By \cite[Theorem~12.1(b)]{Ful98}, we have
\[
\deg(s^*[C_{b \tilde D_1 \cap \ldots \cap b \tilde D_d} Y]) \geq \sum_{i \in \{1, \ldots, q\}} \deg_L(W_i) \geq 0,
\]
where $W_1, \ldots, W_q$ are the distinct irreducible (possibly nonreduced) components of $s^{-1}(C_{b \tilde D_1 \cap \ldots \cap b \tilde D_d} Y)$, where $q$ is a nonnegative integer.

Lastly, note that if $\tilde D_1 \cap \ldots \cap \tilde D_d$ is empty, then $(b \tilde D_1) \cdot \ldots \cdot (b \tilde D_d)$ is a cycle class on the empty scheme, therefore its degree is zero. Conversely, if $(b \tilde D_1) \cdot \ldots \cdot (b \tilde D_d)$ is zero, then $s^{-1}(C_{b \tilde D_1 \cap \ldots \cap b \tilde D_d} Y)$ is empty, therefore $\tilde D_1 \cap \ldots \cap \tilde D_d$ is empty.
\end{proof}

\begin{remark}
If $\bm w = (1, \ldots, 1)$ in Theorem~\ref{thm:generalization of Fulton for a smooth variety}, then we recover the classical intersection theory result. There is an error in the formula in \cite[Corollary in Example~12.4.8]{Ful98}, namely the factor $1 / {[\kappa(P) : \Bbbk]}$ is missing in front of $\sum_{i=1}^t m_i \deg(\alpha_i)$ and $\sum_{i=1}^t m_i \deg(Z_i)$.
The point $P$ is $\Bbbk$-rational in \cite[Corollary~12.4]{Ful98} and \cite[Theorem~3]{FL82}, so that $[\kappa(P) : \Bbbk] = 1$.
\end{remark}

\subsection{Intersection cycles of mobile linear systems around LCI singularities} \label{sec:LCI singularities}

Let $c \ge 0$, $d \ge 2$ be integers and let $\Bbbk$ be a field.
Let $P \in \mathcal{X}$ be the germ of a regular point on an algebraic $\Bbbk$-scheme of dimension $d+c$ with local coordinates $x_1, \dots, x_{d+c}$ at $P$, and let
\[
P \in X := (f_1 = \cdots = f_{c} = 0) \subset \mathcal{X}
\]
be a local complete intersection singularity of dimension $d$, where $f_1, \dots, f_c \in \mathcal{O}_{\mathcal{X},P}$ is a regular sequence.
We assume that $X$ is normal, that is, it is regular in codimension~$1$.
We choose a tuple ${\bm w} := (w_1, \dots, w_{d+c})$ of positive integers such that $\gcd {\bm w} := \gcd (w_1, \dots, w_{d+c}) = 1$.
We consider the ${\bm w}$-weights on the coordinates $x_1, \dots, x_{d+c}$, that is,
\[
\operatorname{wt} (x_1, \dots, x_{d+c}) = {\bm w} = (w_1, \dots, w_{d+c}).
\]

Let $\Phi \colon \mathcal{Y} \to \mathcal{X}$ be the weighted blow-up of $\mathcal{X}$ at $P \in \mathcal{X}$ with weights $\operatorname{wt} (x_1, \dots, x_{d+c}) = {\bm w}$ and let $\mathcal{E}$ be its exceptional divisor.
Let $Y$ be the strict transform of $X$ via $\Phi$ and let $\varphi$ be the restriction $\Phi|_Y$.
Let $E$ be the exceptional divisor of $\varphi$.
We consider the following condition.

\begin{condition} \label{cond:normalwbl}
The closed subscheme
\begin{equation} \label{eq:normalwbl}
(f^{\bm w}_1 = \cdots = f^{\bm w}_c = 0) \subset \mathbb{P} (\bm w)
\end{equation}
is an integral scheme of dimension $d-1$, where $\mathbb{P} ({\bm w}) = \mathbb{P} (w_1, \dots, w_{d+c})$ is the weighted projective space with homogeneous coordinates $x_1, \dots, x_{d+c}$ of weights $w_1, \dots, w_{d+c}$, respectively.
\end{condition}

\begin{lemma} \label{lem:lciexceptional}
Under the above setting, we assume that \emph{Condition \ref{cond:normalwbl}} is satisfied.
Then the following assertions hold.
\begin{enumerate}
\item $\mathcal{E}|_Y = E$ and it is a prime divisor.
\item The variety $Y$ is normal.
\item For any Cartier divisor $D$ on $X$, there is a Cartier divisor $\mathcal{D}$ on $\mathcal{X}$ such that $\mathcal{D}|_X = D$ and $v_{\mathcal{E}} (\mathcal{D}) = v_E (D)$.
\end{enumerate}
\end{lemma}

\begin{proof}
The scheme-theoretic intersection $\mathcal{E} \cap Y$ is isomorphic to the closed subscheme defined in \eqref{eq:normalwbl} and $E$ coincides with them set-theoretically.
The scheme $\mathcal{E} \cap Y$ is integral by Condition \ref{cond:normalwbl}.
It follows that we have $\mathcal{E}|_Y = E$ as divisors on $Y$.
This proves (1).

We prove (2).
The variety $Y$ is covered by Zariski open subsets $U_i$ for $i = 1, \dots, d+c$, where $U_i$ is the quotient of a local complete intersection variety which is regular in codimension $1$ by an action of the cyclic group $\boldsymbol{\mu}_{w_i}$.
In particular, $Y$ is normal since the quotient of a normal variety is normal.

We prove (3).
We may assume that $D$ is effective. Let $\phi \in \mathcal{O}_{X,P}$ be the local equation of $D$ and let $g \in \mathcal{O}_{\mathcal{X},P}$ be such that its residue class in $\mathcal{O}_{X,P} \cong \mathcal{O}_{\mathcal{X},P}/(f_1,\dots,f_c)$ is $\phi$.
Possibly replacing $g$, we may assume that $g^{\bm w} \notin (f_1^{\bm w}, \dots, f_c^{\bm w})$.
We set $\mathcal{D} := \mathrm{div} (g)$, which is a Cartier divisor on $\mathcal{X}$.
We have
\begin{equation} \label{eq:ineqliclem}
\Phi^*\mathcal{D} = \Phi_*^{-1}\mathcal{D} + v_{\mathcal{E}} (\mathcal{D}) \mathcal{E}.
\end{equation}
The restriction $(\Phi_*^{-1}\mathcal{D})|_{\mathcal{E}}$ is isomorphic to the closed subscheme $(g^{\bm w} = 0)$ in $\mathcal{E}$, where we identify $\mathcal{E}$ with the weighted projective space $\mathbb{P} ({\bm w})$ with homogeneous coordinates $x_1, \dots, x_{d+c}$ of weights $w_1, \dots, w_{d+c}$, respectively.
It follows that $\Phi_*^{-1}\mathcal{D}$ does not contain $E$, and hence $(\Phi_*^{-1}\mathcal{D})|_Y = \varphi_*^{-1} D$.
By restricting the equality \eqref{eq:ineqliclem} to $Y$, we obtain
\[
\varphi^*D = \varphi_*^{-1} D + v_{\mathcal{E}} (\mathcal{D}) E
\]
since $\mathcal{E}|_Y = E$.
This shows $v_{\mathcal{E}} (\mathcal{D}) = v_E (D)$.
\end{proof}

\begin{theorem} \label{thm:ineqmobLCI}
Under the above setting, we assume that $w_1 \le w_2 \le \cdots \le w_{d+c}$ and that \emph{Condition \ref{cond:normalwbl}} is satisfied.
Suppose that there is a mobile linear system $\mathcal{M}$ of Cartier divisors on $X$ and a rational number $n > 0$ such that the pair $(X, \frac{1}{n} \mathcal{M})$ is not canonical at $E$.
Then, for general members $D_1, D_2 \in \mathcal{M}$, we have
\[
\operatorname{mult}_P (D_1 \cdot D_2) > \frac{w (f_1) \cdots w (f_c) a_E(X)^2}{w_{d-1} \cdots w_{d+c}} n^2
\]
\end{theorem}

\begin{proof}
By Lemma~\ref{lem:lciexceptional}, we can take divisors $\mathcal{D}_i$ on $\mathcal{X}$ such that $\mathcal{D}_i|_X = D_i$ and $v_{\mathcal{E}} (\mathcal{D}_i) = v_E (D_i)$ for $i = 1, 2$.
As a $(d-2)$-cycle on $\mathcal{X}$, we have
\[
D_1 \cdot D_2 = \mathcal{D}_1 \cdot \mathcal{D}_2 \cdot X = \mathcal{D}_1 \cdot \mathcal{D}_2 \cdot \mathcal{X}_1 \cdot \ldots \cdot \mathcal{X}_c,
\]
where $\mathcal{X}_i := (f_i = 0) \subset \mathcal{X}$.
Let $\mathcal{H}_1, \dots, \mathcal{H}_{d-2}$ be general hyperplanes on $\mathcal{X}$ passing through $P$ so that $P$ is an isolated component of
\[
\mathcal{D}_1 \cap \mathcal{D}_2 \cap X \cap \mathcal{H}_1 \cap \cdots \cap \mathcal{H}_{d-2}.
\]
We may assume that $v_{\mathcal{E}} (\mathcal{H}_i) \ge w_i$ for $i = 1, \dots, d-2$.
By the assumption that $(X, \frac{1}{n} \mathcal{M})$ is not canonical at $E$, we obtain
\[
v_{\mathcal{E}} (\mathcal{D}_i) = v_E (D_i) > n a_E(X).
\]
Thus, we have
\[
\begin{split}
\operatorname{mult}_P (D_1 \cdot D_2) &= \operatorname{mult}_P (\mathcal{D}_1 \cdot \mathcal{D}_2 \cdot \mathcal{X}_1 \cdot \ldots \cdot \mathcal{X}_c) \\
&= \operatorname{mult}_P (\mathcal{D}_1 \cdot \mathcal{D}_2 \cdot \mathcal{X}_1 \cdot \ldots \cdot \mathcal{X}_c \cdot \mathcal{H}_1 \cdot \ldots \cdot \mathcal{H}_{d-2}) \\
&> \frac{(n a_E (X))^2 w (f_1) \cdots w (f_c) \cdot w_1 \cdots w_{d-2}}{w_1 \cdots w_{d+c}} \\
&= \frac{w (f_1) \cdots w (f_c) a_E(X)^2}{w_{d-1} \cdots w_{d+c}} n^2,
\end{split}
\]
where the inequality follows from Theorem~\ref{thm:generalization of Fulton for a smooth variety}.
\end{proof}

We have a formula for computing the discrepancy $a_E (X) := \operatorname{ord}_E K_{Y/X}$ under some additional assumptions.

\begin{lemma}
\label{lem:discrepancywbl}
Under the above setting, we assume that \emph{Condition \ref{cond:normalwbl}} is satisfied.
Assume in addition that either $c = 0$ or $c \ge 1$ and the weighted complete intersection \eqref{eq:normalwbl} is well-formed (See \S \ref{sec:WPV} for the definition of well-formedness).
Then we have
\[
a_E (X) = \sum\limits_{i=1}^{d+c} w_i - \sum\limits_{i=1}^c w (f_i) - 1.
\]
\end{lemma}

\begin{proof}
We set $\mathcal{X}_i := (f_i = 0) \subset \mathcal{X}$ and $X_i := \mathcal{X}_1 \cap \cdots \cap \mathcal{X}_i$ for $i = 1, \dots, c$.
We also set $\varphi_i = \Phi|_{Y_i} \colon Y_i \to X_i$, where $Y_i$ be the proper transform of $X_i$ on $\mathcal{Y}$.
Let $E_i$ be the $\varphi_i$-exceptional divisor.
We set $X_0 := \mathcal{X}$, $Y_0 := \mathcal{Y}$, $\varphi_0 := \Phi \colon Y_0 \to X_0$ and $E_0 := \mathcal{E}$.
We have the formula
\begin{equation} \label{eq:discrepancywbl}
K_{Y_0} = \Phi^*K_{X_0} + \left(\sum_{i=1}^{d+c} w_i - 1\right) E_0,
\end{equation}
which in particular proves the assertion when $c = 0$.
We proceed the proof by assuming $c \ge 1$.

We choose and fix any $i \in \{0, 1, \dots, c\}$.
We see that the closed subscheme
\begin{equation} \label{eq:wciE_i}
(f_1^{{\bm w}} = \cdots = f_i^{{\bm w}} = 0) \subset \mathbb{P} ({\bm w})
\end{equation}
is integral because otherwise the closed subscheme \eqref{eq:normalwbl} cannot be integral and this is impossible by Condition \ref{cond:normalwbl}.
We can apply Lemma~\ref{lem:lciexceptional} to the local complete intersection singularity $P \in X_i$ and we see that $E_i = \mathcal{E}|_{Y_i}$ is a prime divisor that coincides with the weighted complete intersection \eqref{eq:wciE_i} and that $Y_i$ is a normal variety.
The divisor $X_{i+1}$ is a Cartier divisor on $X_i$ for any $i$.
The exceptional divisor $\mathcal{E}$ is a Cartier divisor on $\mathcal{Y} \setminus \operatorname{Sing} (\mathcal{E})$, and hence $E_i = \mathcal{E} \cap Y_i$ is a Cartier divisor on $Y_i \setminus \operatorname{Sing} (\mathcal{E}) \cap E_i$.
The closed set $\operatorname{Sing} (\mathcal{E}) \cap E_i$ is of codimension at least $3$ in $Y_i$ since $E_i \subset \mathcal{E} \cong \mathbb{P} ({\bm w})$ is well-formed.
It follows that $Y_{i+1} = \varphi_i^*X_i - {\bm w} (f_i) E_i$ is a Cartier divisor away from a closed subset of codimension at least $3$ in $Y_i$.
By the adjunction formula, we have $(K_{X_i} + X_{i+1})|_{X_{i+1}} = K_{X_{i+1}}$ and $(K_{Y_i} + Y_{i+1})|_{Y_{i+1}} = K_{Y_{i+1}}$.
By applying adjunction formulae successively for $X = X_c \subset \cdots \subset X_0$ and $Y = Y_c \subset \cdots \subset Y_0$ starting with \eqref{eq:discrepancywbl}, we obtain the desired formula.
\end{proof}

\subsection{\texorpdfstring
    {Local inequality for $cA_k$-singularities}
    {Local inequality for cA\_k-singularities}
}

In this section, the ground field is assumed to be $\mathbb{C}$.

\begin{theorem} \label{thm:locineqwblcA}
Let $P \in X$ be an isolated $cA_k$ singularity, where $k \ge 1$, let $\mathcal M$ be a mobile linear system of Cartier divisors on $X$ and let $n > 0$ be a rational number.
Let $\varphi \colon Y \to X$ be a divisorial contraction with center $P$ and let $D_1, D_2 \in \mathcal{M}$ be general members.
Suppose that the pair $(X, \frac{1}{n} \mathcal{M})$ is not canonical at the exceptional divisor $E$.
\begin{enumerate}
\item If $\varphi$ is of non-exceptional type and its weight is $(r_1, r_2, a, 1)$, where $\gcd(a, r_1) = \gcd(a, r_2) = 1$ and $r_1 + r_2 = (k+1) a$, then
    \[
    \operatorname{mult}_P (D_1 \cdot D_2) > \frac{(r_1 + r_2)^2}{(k+1) r_1 r_2} n^2.
    \]
\item If $k = 1$ and $\varphi$ is of exceptional type, then
\[
\operatorname{mult}_P (D_1 \cdot D_2) > \frac{16}{5} n^2.
\]
\item If $k = 2$ and $\varphi$ is of exceptional type, then
\[
\operatorname{mult}_P (D_1 \cdot D_2) > \frac{9}{4} n^2.
\]
\end{enumerate}
\end{theorem}

\begin{proof}
We use Lemma~\ref{lem:truncate} to move to a local analytic coordinate system.
The weights of $\varphi$ is $(1, 5, 3, 2)$ (resp.\ $(4, 3, 2, 1)$) if we are in case (2) (resp.\ (3)).
In any case, Condition~\ref{cond:normalwbl} is satisfied and the assertion follows from Theorem~\ref{thm:ineqmobLCI}.
\end{proof}

\begin{theorem}[= Theorem~\ref{MainThm}] \label{body-MainThm}
\maintheorem{}
\end{theorem}

\begin{proof}
For any positive integers $r_1$ and $r_2$, we have
\[
\frac{(r_1+r_2)^2}{(k+1)r_1r_2} n^2 \ge \frac{4}{k+1} n^2.
\]
If $k = 1$ (resp.\ $2$), then $16/5n^2 > 2 n^2$ (resp.\ $9/4n^2 > 4/3 n^2$).
Thus, the assertion follows from Theorem~\ref{thm:locineqwblcA}.
\end{proof}

\subsection{Algebraic corollary}

Corollary~\ref{thm:algebraic inequality} is an algebraic corollary of Theorem~\ref{thm:generalization of Fulton for a smooth variety}, similar to \cite[Example~12.4.9]{Ful98}.

\begin{corollary} \label{thm:algebraic inequality}
Let $d$ be a positive integer.
Let $\Bbbk$ be an arbitrary field and let ${\bm w} = (w_1, \ldots, w_d)$ be a $d$-tuple of positive integers such that $\gcd {\bm w} = 1$.
We consider the weights $\operatorname{wt} (x_1, \dots, x_d) = {\bm w}$ and let $f_1, \ldots, f_d \in \Bbbk [[x_1, \ldots, x_d]]$.
Then
\[
\dim_{\Bbbk} \frac{\Bbbk [[x_1, \ldots, x_d]]}{(f_1, \ldots, f_d)}
\geq \frac{{\bm w} (f_1) \cdot \ldots \cdot {\bm w} (f_d)}{w_1 \cdot \ldots \cdot w_d}.
\]
Moreover, if $f_1, \ldots, f_d$ are polynomials and
\[
\dim_{\Bbbk} \frac{\Bbbk [[x_1, \ldots, x_d]]}{(f_1, \ldots, f_d)}
\]
is finite, then
\[
\dim_{\Bbbk} \frac{\Bbbk [[x_1, \ldots, x_d]]}{(f_1, \ldots, f_d)}
\leq \frac{\deg_{\bm w}(f_1) \cdot \ldots \cdot \deg_{\bm w}(f_d)}{w_1 \cdot \ldots \cdot w_d},
\]
where $\deg_{\bm w}(f)$ denotes the ${\bm w}$-degree of a nonzero polynomial $f \in \Bbbk [x_1, \ldots, x_d]$, meaning the greatest integer $k$ such that the quasihomogeneous weighted-degree $k$ part of $f$ is nonzero.
\end{corollary}

\begin{proof}
First, we prove the lower bound.
If the left-hand side is infinite, then the inequality holds trivially.
Otherwise, there exists a positive integer $N$ such that $\mathfrak{m}^N \subseteq (f_1, \ldots, f_d)$, where $\mathfrak{m}$ is the maximal ideal $(x_1, \ldots, x_d)$.
Therefore $\mathfrak{m}^N \subseteq (g_1, \ldots, g_d)$, where the $g_i$ are the truncations $(f_i)_{\deg \leq N} \in \Bbbk [x_1, \ldots, x_k]$ up to degree~$N$.
We find ${\bm w} (g_i) \geq {\bm w} (f_i)$ and
\[
\dim_{\Bbbk} \frac{\Bbbk [[x_1, \ldots, x_d]]}{(f_1, \ldots, f_d)} = \dim_{\Bbbk} \frac{\Bbbk [[x_1, \ldots, x_d]]}{(g_1, \ldots, g_d)}.
\]
Let $P$ be the origin in $X := \mathbb A_{\Bbbk}^d$ and let $D_i := V(g_i)$. By \cite[Example~7.1.10(b)]{Ful98},
\[
\dim_{\Bbbk} \frac{\Bbbk [[x_1, \ldots, x_d]]}{(g_1, \ldots, g_d)} = \operatorname{mult}_P(D_1 \cdot \ldots \cdot D_d).
\]
The first inequality in the statement of Corollary~\ref{thm:algebraic inequality} now follows from Theorem~\ref{thm:generalization of Fulton for a smooth variety}.

Next, we prove the upper bound.
Assign weights $1, w_1, \ldots, w_d$ respectively to the variables $x_0, \ldots, x_d$.
Let $F_1, \ldots, F_d \in \Bbbk [x_0, \ldots, x_d]$ be the quasihomogeneous polynomials such that for every $i \in \{1, \ldots, d\}$, $x_0$ does not divide $F_i$ and $F_i(1, x_1, \ldots, x_d) = f_i$. Let $V(F_i)$ denote the subscheme in the weighted projective space $\mathbb P(1, w_1, \ldots, w_d)$ defined by~$F_i$. Let $P$ be the point $[1, 0, \ldots, 0]$ in $\mathbb P(1, w_1, \ldots, w_d)$. Note that either $P$ is not in $V(F_1, \ldots, F_d)$ or $P$ is an isolated point of $V(F_1, \ldots, F_d)$. By \cite[Example~7.1.10(b)]{Ful98},
\[
\dim_{\Bbbk} \frac{\Bbbk [[x_1, \ldots, x_d]]}{(f_1, \ldots, f_d)} = \operatorname{mult}_P\bigl(V(F_1) \cdot \ldots \cdot V(F_d)\bigr).
\]
Since each $V(F_i)$ is $\mathbb Q$-linearly equivalent to $\frac{\operatorname{deg}_{\bm w}(f_i)}{w_i} \cdot V(x_i)$, we find
\[
V(F_1) \cdot \ldots \cdot V(F_d) = \frac{\deg_{\bm w}(f_1) \cdot \ldots \cdot \deg_{\bm w}(f_d)}{w_1 \cdot \ldots \cdot w_d}.
\]
It is left to prove that on the open subset $U := \mathbb P(1, w_1, \ldots, w_d) \setminus \{P\}$, the degree of the intersection class $(V(F_1) \cap U) \cdot \ldots \cdot (V(F_d) \cap U)$ is nonnegative. As in the proof of inequality~\ref{eqn:weighted mult inequality} in Theorem~\ref{thm:generalization of Fulton for a smooth variety}, this follows from the arguments in the proof of \cite[Theorem~3]{FL82} by applying \cite[Theorem~2(B)]{FL82}.
\end{proof}

\section{Generalizations to (LCI-)quotient singularities}

Throughout this section, we work over the field $\mathbb{C}$ of complex numbers.

\subsection{Main inequality for cyclic quotient singularities} \label{sec:Mainineqquotient}

We consider a generalization of Theorem~\ref{thm:generalization of Fulton for a smooth variety} to cyclic quotient singularities.
Suppose that $d \ge 2$.
Let $r > 0$ and $a_1, \dots, a_d \ge 0$ be integers.

We say that an action of $\boldsymbol{\mu}_r$ on a germ $P \in V$ of a smooth $d$-dimensional variety is \textit{of type} $\frac{1}{r} (a_1, \dots, a_d)$ if there exist a regular system of parameters $x_1, \dots, x_d$ of $V$ at $P$ such that the $\mathbb C$-algebra automorphism of $\mathcal{O}_{V,P}$ induced by $\zeta \in \boldsymbol{\mu}_r$ sends $x_i$ to $\zeta^{a_i} x_i$ for $i = 1, \dots, d$ and for any $r$th root of unity $\zeta \in \boldsymbol{\mu}_r$, so that the $\mathbb C$-algebra isomorphism $\mathbb{C} [[x_1, \dots, x_d]] \to \hat{\mathcal{O}}_{V,P}$ induced by the inclusion $\mathbb C [x_1, \dots, x_n] \to \mathcal{O}_{V,P}$ is $\boldsymbol{\mu}_r$-equivariant, where the $\boldsymbol{\mu}_r$-action on $\mathbb{C} [[x_1, \dots, x_n]]$ is defined by $x_i \mapsto \zeta^{a_i} x_i$ for $i = 1, \dots, d$ and for any $r$th root of unity $\zeta \in \mathbb{C}$.

Let $P \in X$ be a germ of a cyclic quotient singular point of type $\frac{1}{r} (a_1, \dots, a_d)$, that is, the germ $P \in X$ is isomorphic to the germ
\[
P' \in \hat{X}/\boldsymbol{\mu}_r,
\]
where $\hat{P} \in \hat{X}$ is a germ of a smooth $d$-dimensional variety admitting a $\boldsymbol{\mu}_r$-action of type $\frac{1}{r} (a_1, \dots, a_d)$ and $P'$ is the image of $\hat{P}$ by the quotient morphism $\hat{X} \to \hat{X}/\boldsymbol{\mu}_r$.
We identify $P \in X$ with $P' \in \hat{X}/\boldsymbol{\mu}_r$ and denote by $q_X \colon \hat{X} \to X$ the quotient morphism.

We assume that
\[
\gcd(r, a_1, \ldots, a_{j-1}, a_{j+1}, \ldots, a_d) = 1
\]
for every $j \in \{1, \ldots, d\}$, that is, $q_X \colon \hat{X} \to X$ is \'{e}tale in codimension $1$.
By convention, we allow $r = 1$, corresponding to a smooth point $P \in X$.

Let ${\bm w} = (w_1, \dots, w_d)$ be a tuple of positive integers such that $w_i \equiv a_i \pmod{r}$ for $i = 1, \dots, d$ and let $x_1, \dots, x_d$ be a regular system of parameters of $\hat{X}$ at $\hat{P}$ with respect to which the action of $\boldsymbol{\mu}_r$ on $\hat{X}$ is linear of type $\frac{1}{r} (a_1, \dots, a_d)$.
We give a description of weighted blowup $\varphi \colon Y \to X$ of $P \in X$ with weights $w_1/r, \dots, w_d/r$.
Let $\hat{\varphi} \colon \hat{Y} \to \hat{X}$ be the ${\bm w}$-blowup of $\hat{P} \in \hat{X}$ with respect to the regular system of parameters $x_1, \dots, x_d$ and let $\hat{E} \cong \mathbb{P} (w_1, \dots, w_d)$ be its exceptional divisor.
The $\boldsymbol{\mu}_r$-action on $\hat{X}$ lifts to an action on $\hat{Y}$ and we denote by $Y$ its quotient $\hat{Y}/\boldsymbol{\mu}_r$.
Then the restriction $\varphi := \hat{\varphi}|_Y \colon Y \to X$ is called a \textit{weighted blowup} of $P \in X$ with weights $w_1/r, \dots, w_d/r$ with respect to $x_1, \dots, x_d$.
Let $q_Y \colon \hat{Y} \to Y$ be the quotient morphism, which is a finite morphism of degree $r$ ramified along the $\varphi$-exceptional divisor $E$, and we have $q_Y^*E = r \hat{E}$.
We obtain the diagram:
\[
\xymatrix{
\hat{Y} \ar[d]_{\hat{\varphi}} \ar[r]^{q_Y} & Y \ar[d]^{\varphi} \\
\hat{X} \ar[r]_{q_X} & X}
\]

\begin{corollary} \label{thm:generalization of Fulton for a quotient singularity}
Let the notation and assumptions be as above.
Let $D_1, \ldots, D_d$ be effective $\mathbb{Q}$-divisors on $X$ such that $P$ is an isolated point of $D_1 \cap \ldots \cap D_d$.
For every~$i$, let $\tilde D_i$ be the strict transform of $D_i$ on $Y$.
Then, both of the following hold:
\begin{align*}
    \operatorname{mult}_P(D_1 \cdot \ldots \cdot D_d) & = \frac{v_E(D_1) \cdot \ldots \cdot v_E(D_d) \cdot r^{d-1}}{w_1 \cdot \ldots \cdot w_d} + \deg_E(\tilde D_1 \cdot \ldots \cdot \tilde D_d),\\
    \deg_E(\tilde D_1 \cdot \ldots \cdot \tilde D_d) & \geq 0,
\end{align*}
where $\deg_E(\tilde D_1 \cdot \ldots \cdot \tilde D_d)$ denotes the degree of the $0$-cycle class $\tilde D_1 \cdot \ldots \cdot \tilde D_d$ on a small neighborhood of $E$. Moreover, $\deg_E(\tilde D_1 \cdot \ldots \cdot \tilde D_d)$ is zero if and only if $\tilde D_1 \cap \ldots \cap \tilde D_d \cap E$ is empty.
\end{corollary}

\begin{proof}
For every $i$, let $\hat{D}_i$ be the pullback of $D_i$ under $q_X$.
By Theorem~\ref{thm:generalization of Fulton for a smooth variety},
\begin{align*}
    \operatorname{mult}_{\hat{P}}(\hat{D}_1 \cdot \ldots \cdot \hat{D}_d) & = \frac{{v_{\hat{E}}}(\hat{D}_1) \cdot \ldots \cdot v_{\hat{E}}(\hat{D}_d)}{w_1 \cdot \ldots \cdot w_d} + \deg_{\hat{E}}(\hat{\varphi}_*^{-1} \hat{D}_1 \cdot \ldots \cdot \hat{\varphi}_*^{-1} \hat{D}_d),\\
    \deg_{\hat{E}}(\hat{\varphi}_*^{-1} \hat{D}_1 \cdot \ldots \cdot \hat{\varphi}_*^{-1} \hat{D}_d) & \geq 0,
\end{align*}
with equality if and only if $\hat{\varphi}_*^{-1} \hat{D}_1 \cap \ldots \cap \hat{\varphi}_*^{-1} \hat{D}_n \cap \hat{E}$ is empty. By \cite[section~1.4]{Ful98},
\[
\begin{aligned}
  \operatorname{mult}_{\hat{P}} (\hat{D}_1 \cdot \ldots \cdot \hat{D}_d) & = r \cdot \operatorname{mult}_P(D_1 \cdot \ldots \cdot D_d),\\
  \deg_{\hat{E}}(\hat{\varphi}_*^{-1} \hat{D}_1 \cdot \ldots \cdot \hat{\varphi}_*^{-1} \hat{D}_d) & = r \cdot \deg_{E}(\tilde D_1 \cdot \ldots \cdot \tilde D_d).
\end{aligned}
\]
We have $q_Y^*E = r \hat{E}$ as explained above.
This implies
\[
v_{\hat{E}}(\hat{D}_i) = r \cdot v_E(D_i).
\]
The corollary follows.
\end{proof}

\subsection{Intersection cycles of mobile linear systems around LCI-quotient singularities}

We consider a generalization of Theorem~\ref{thm:ineqmobLCI} to cyclic quotients of local complete intersection singularities.

Let $c \ge 0$, $d \ge 2$, $r > 0$ and $a_1, \dots, a_{d+c} \ge 0$ be integers.
Let $\hat{P} \in \hat{\mathcal{X}}$ be the germ of a smooth $(d+c)$-dimensional variety with local coordinates $x_1, \dots, x_{d+c}$ at $\hat{P}$ and let
\[
\hat{P} \in \hat{X} := (f_1 = \cdots = f_c = 0) \subset \hat{\mathcal{X}}
\]
be a local complete intersection singularity of dimension $d$, where $f_1, \dots, f_c \in \mathcal{O}_{\hat{\mathcal{X}}, \hat{P}}$ is a regular sequence.
We consider the action of $\boldsymbol{\mu}_r$ of type $\frac{1}{r} (a_1, \dots, a_{d+c})$ on $\hat{\mathcal{X}}$ with respect to the local coordinates $x_1, \dots, x_{d+c}$.
We assume that $f_i$ is semi-invariant with respect to the $\boldsymbol{\mu}_r$-action for $i = 1, \dots, c$, so that $\boldsymbol{\mu}_r$ acts on $\hat{X}$.
Let $X$ and $\mathcal{X}$ be the quotients of $\hat{X}$ and $\hat{\mathcal{X}}$ by the $\boldsymbol{\mu}_r$-action.
We assume that $\hat{X}$ is normal, which implies that $X$ is also normal.
We denote by $q_X \colon \hat{X} \to X$ and $q_{\mathcal{X}} \colon \hat{\mathcal{X}} \to \mathcal{X}$ be the quotient morphisms.
We set $P := q_X (\hat{P}) = q_{\mathcal{X}} (\hat{P}) \in X \subset \mathcal{X}$.
We assume that
\[
\gcd(r, a_1, \ldots, a_{j-1}, a_{j+1}, \ldots, a_{d+c}) = 1
\]
for every $j \in \{1, \ldots, d+c\}$.
If $c = 0$, then $\hat{X} = \hat{\mathcal{X}}$ and the singularity $P \in X$ is simply a quotient singularity of type $\frac{1}{r} (a_1, \dots, a_d)$.
We allow $r = 1$, corresponding to the case of a local complete intersection singularity.

Let $w_1, \dots, w_{d+c}$ be positive integers such that $\gcd (w_1, \dots, w_{d+c}) = 1$ and $w_i \equiv a_i \pmod{r}$ for $i = 1, \dots, d+c$. We set ${\bm w} := (w_1, \dots, w_{d+c})$ and ${\bm w}_r := \frac{1}{r} {\bm w} = \frac{1}{r} (w_1, \dots, w_{d+c})$.
Let $\varphi \colon Y \to X$ be a weighted blow-up of $P \in X$ with weights $\operatorname{wt} (x_1, \dots, x_{d+c}) = {\bm w}_r$.
Let $\hat{\Phi} \colon \hat{\mathcal{Y}} \to \hat{\mathcal{X}}$ be the weighted blow-up of $\hat{\mathcal{X}}$ at the origin $\hat{P}$ with weights $\operatorname{wt} (x_1, \dots, x_{d+c}) = {\bm w}$.
The $\boldsymbol{\mu}_r$-action on $\hat{\mathcal{X}}$ lifts to $\hat{\mathcal{Y}}$, and we obtain the induced morphism $\Phi \colon \mathcal{Y} \to \mathcal{X}$, where $\mathcal{Y}$ and $\mathcal{X}$ are the quotients of $\hat{\mathcal{Y}}$ and $\hat{\mathcal{X}}$, respectively, by the $\boldsymbol{\mu}_r$-action.
We have the quotient morphism $q_{\mathcal{Y}} \colon \hat{\mathcal{Y}} \to \mathcal{Y}$, and the commutative diagram:
\begin{equation} \label{eq:wbldiagamb}
\xymatrix{
\hat{\mathcal{Y}} \ar[d]_{\hat{\Phi}} \ar[r]^{q_{\mathcal{Y}}} & \mathcal{Y} \ar[d]^{\Phi} \\
\hat{\mathcal{X}} \ar[r]_{q_{\mathcal{X}}} & \mathcal{X}
}
\end{equation}
We can identify $\hat{Y}$ and $Y$ with the proper transforms $\hat{\Phi}_*^{-1} \hat{X} \subset \hat{\mathcal{Y}}$ and $\Phi_*^{-1} X \subset \mathcal{Y}$, respectively.
Under the above identifications, $\hat{\varphi}$ and $\varphi$ coincide with the restrictions $\hat{\Phi}|_{\hat{Y}}$ and $\Phi|_Y$.
The diagram \eqref{eq:wbldiagamb} restricts to the diagram:
\begin{equation} \label{eq:wbldiag}
\xymatrix{
\hat{Y} \ar[d]_{\hat{\varphi}} \ar[r]^{q_Y} & Y \ar[d]^{\varphi} \\
\hat{X} \ar[r]_{q_X} & X,
}
\end{equation}
where $q_Y := q_{\mathcal{Y}}|_{\hat{Y}} \colon \hat{Y} \to Y$ is the quotient morphism under the induced $\boldsymbol{\mu}_r$-action.
Let $\hat{\mathcal{E}}$ and $\mathcal{E}$ be the exceptional divisors of $\hat{\Phi}$ and $\Phi$, respectively, both of which are isomorphic to $\mathbb{P} ({\bm w})$.
Note here that, by a slight abuse of notation, $\mathbb{P} ({\bm w}) = \mathbb{P} (w_1, \dots, w_{d+c})$ is thought of as the weighted projective space with homogeneous coordinates $x_1, \dots, x_{d+c}$ of weights $w_1, \dots, w_{d+c}$, respectively.
The quotient morphism $q_{\mathcal{Y}} \colon \hat{\mathcal{Y}} \to \mathcal{Y}$ is a finite morphism ramified along $\mathcal{E}$ and we have $q_{\mathcal{Y}}^*\mathcal{E} = r \hat{\mathcal{E}}$.
We denote by $\hat{E}$ and $E$ the exceptional divisors of $\hat{\varphi}$ and $\varphi$, respectively.
For an element $f \in \mathcal{O}_{\mathcal{X},P}$, we denote by ${\bm w} (f)$ (resp.\ ${\bm w}_r (f)$) the weighted order of $f$ with respect to the ${\bm w}$-weights (resp.\ ${\bm w}_r$-weights) and by $f^{\bm w}$ the least weight part of $f$ with respect to the ${\bm w}$-weights in the sense of Definition \ref{def:weight}.

\begin{lemma}
Under the above setting, we assume that \emph{Condition~\ref{cond:normalwbl}} is satisfied.
Then the following assertions hold.
\begin{enumerate}
\item $\hat{\mathcal{E}}|_{\hat{Y}} = \hat{E}$ and $\mathcal{E}|_Y = E$, and they are both prime divisors.
\item $Y$ is a normal variety.
\item $\hat{E}$ is isomorphic to $E$ and we have $q_Y^*E = r \hat{E}$.
\item If in addition either $c = 0$ or $c \ge 1$ and the weighted complete intersection \eqref{eq:normalwbl} is well-formed, then
\[
a_E(X) = \dfrac{1}{r} \left(\sum\limits_{i=1}^{d+c} w_i - \sum\limits_{i=1}^c {\bm w} (f_i)\right) - 1.
\]
\end{enumerate}
\end{lemma}

\begin{proof}
By Lemma~\ref{lem:lciexceptional} applied to the local complete intersection singularity $\hat{P} \in \hat{X}$, we see that  $\hat{\mathcal{E}}|_{\hat{Y}} = \hat{E}$ is a prime divisor and $\hat{Y}$ is normal.
These imply that $\mathcal{E}|_Y = E$ and $Y$ is normal.
Taking restriction of $q_{\mathcal{Y}}^* \mathcal{E} = r \hat{\mathcal{E}}$ to $Y$, we obtain $q_Y^*E = r \hat{E}$.

By Lemma~\ref{lem:discrepancywbl} applied to $\hat{P} \in \hat{X}$, we have
\[
K_{\hat{Y}} = \hat{\varphi}^*K_{\hat{X}} + \left(\sum_{i=1}^{d+c} w_i - \sum_{i=1}^c {\bm w} (f_i) - 1\right) \hat{E}.
\]
We see that $K_{\hat{X}} = q_X^*K_X$ and
\[
K_{\hat{Y}} = q_Y^*K_Y + \frac{r-1}{r} q_Y^*E = q_Y^*K_Y + (r-1)\hat{E}.
\]
Combining these equalities, we obtain (4).
\end{proof}

\begin{theorem} \label{thm:wbllocintnumineq}
Under the above setting, we assume that $w_1 \le w_2 \le \cdots \le w_{d+c}$ and \emph{Condition \ref{cond:normalwbl}} is satisfied.
Suppose that there is a mobile linear system $\mathcal{M}$ on $X$ and a rational number $n > 0$ such that $q_X^*\mathcal{M}$ is a linear system of Cartier divisors on $\hat{X}$ and the pair $(X, \frac{1}{n} \mathcal{M})$ is not canonical at the $\varphi$-exceptional divisor $E$.
Then, for general members $D_1, D_2$ of $\mathcal{M}$, we have
\[
\operatorname{mult}_P (D_1 \cdot D_2) > \frac{r {\bm w} (f_1) \cdots {\bm w} (f_c) a_E(X)^2}{w_{d-1} w_d \cdots w_{d+c}} n^2.
\]
\end{theorem}

\begin{proof}
Let $D_1, D_2 \in \mathcal{M}$ be general members.
By assumption, $\hat{D}_i := q_X^*D_i$ is a Cartier divisor on $\hat{X}$, and hence we can take divisors $\hat{\mathcal{D}}_i$ on $\hat{\mathcal{X}}$ such that $\hat{\mathcal{D}}_i|_X = \hat{D}_i$ for $i = 1, 2$.
By Lemma~\ref{lem:lciexceptional}, we can choose $\hat{\mathcal{D}}_i$ so that $v_{\hat{E}} (\hat{D}_i) = v_{\hat{\mathcal{E}}} (\hat{\mathcal{D}}_i)$.
As a $(d-2)$-cycle on the germ $\hat{\mathcal{X}}$, we have
\[
\hat{D}_1 \cdot \hat{D}_2 = \hat{\mathcal{D}}_1 \cdot \hat{\mathcal{D}}_2 \cdot \hat{X} = \hat{\mathcal{D}}_1 \cdot \hat{\mathcal{D}}_2 \cdot \hat{\mathcal{X}}_1 \cdots \hat{\mathcal{X}}_c,
\]
where $\hat{\mathcal{X}}_i := (f_i = 0) \subset \hat{\mathcal{X}}$ is a hypersurface for $i = 1, \dots, c$.
Let $\hat{\mathcal{H}}_1, \dots, \hat{\mathcal{H}}_{d-2}$ be general hyperplanes of $\hat{\mathcal{X}}$ passing through $\hat{P}$ such that $\hat{P}$ is an isolated component of
\[
\hat{\mathcal{D}}_1 \cap \hat{\mathcal{D}}_2 \cap \hat{X} \cap \hat{\mathcal{H}}_1 \cap \cdots \cap \hat{\mathcal{H}}_{d-2}.
\]
We may assume that $v_{\hat{\mathcal{E}}} (\hat{\mathcal{H}}_i) \ge w_i$ for $i = 1, 2, \dots, d-2$.
Moreover, we have
\[
v_{\hat{\mathcal{E}}} (\hat{\mathcal{D}}_i) = v_{\hat{E}} (\hat{D}_i) = r v_E (D_i) > r n a_E(X),
\]
where the last inequality follows from the assumption that $(X, \frac{1}{n} \mathcal{M})$ is not canonical at $E$.
By Theorem~\ref{thm:generalization of Fulton for a smooth variety}, we have
\[
\begin{split}
\operatorname{mult}_{\hat{P}} (\hat{D}_1 \cdot \hat{D}_2)
&= \operatorname{mult}_{\hat{P}} (\hat{\mathcal{D}}_1 \cdot \hat{\mathcal{D}}_2 \cdot \hat{\mathcal{X}}_1 \cdots \hat{\mathcal{X}}_c \cdot \hat{\mathcal{H}}_1 \cdots \hat{\mathcal{H}}_{d-2}) \\
&> \frac{(r n a_E(X))^2 \cdot {\bm w} (f_1) \cdots {\bm w} (f_c) \cdot w_1 \cdots w_{d-2}}{w_1 \cdots w_{d+c}} \\
&= \frac{r^2 {\bm w} (f_1) \cdots {\bm w} (f_c) a_E(X)^2}{w_{d-1} w_d \cdots w_{d+c}} n^2.
\end{split}
\]
The desired inequality follows since $\operatorname{mult}_P (D_1 \cdot D_2) = \frac{1}{r} \operatorname{mult}_{\hat{P}} (\hat{D}_1 \cdot \hat{D}_2)$.
\end{proof}

\section{Sextic double solids}
\label{sc:SDS}

The aim of this section is to prove the birational rigidity of factorial sextic double solids with at most $cA_1$-points and ordinary $cA_2$-points.

\subsection{Birational involutions}

This section is entirely devoted to the proof of the following result.

\begin{proposition} \label{prop:DSbirinv}
Let $X$ be a factorial sextic double solid admitting an isolated $cA_2$ point $P \in X$ and let $\varphi \colon Y \to X$ be a divisorial contraction of discrepancy $1$ with center $P$.
Then, either $\varphi$ is not a maximal extraction or there is an elementary self-link $\sigma \colon X \dashrightarrow X$ initiated by $\varphi$.
\end{proposition}

Note that the singularity $P \in X$ is not necessarily ordinary in Proposition~\ref{prop:DSbirinv}.
Let $X$ be a factorial sextic double solid admitting a singular point $P \in X$ of type $cA_2$.
Then, by \cite[Theorem A]{PaemurruSDS}, we may assume that $X \subset \mathbb{P} (1_x, 1_y, 1_z, 1_t, 3_w)$ is defined by an equation of the form
\[
f := -w^2 + x^4 t^2 + x^3 g_3 + x^2 g_4 + x g_5 + g_6 = 0,
\]
where $g_i = g_i (y, z, t)$ is a homogeneous polynomial of degree $i$ such that $g_3 (y, z, 0) \ne 0$, and $P = (1\!:\!0\!:\!0\!:\!0\!:\!0)$.
We note that the expression $\mathbb{P} (1_x, 1_y, 1_z, 1_t, 3_w)$ means that it is the weighted projective space $\mathbb{P} (1, 1, 1, 1, 3)$ with homogeneous coordinates $x, y, z, t, w$ of weights $1, 1, 1, 1, 3$, respectively.
We introduce new homogeneous coordinates $\xi$ and $\zeta$ of weights $3$, and consider the relations:
\[
\xi = w + x^2 t, \quad
\zeta = - w + x^2 t.
\]
Eliminating $w = (\xi - \zeta)/2$, we identify $X$ with the complete intersection of codimension $2$ in $\mathbb{P} := \mathbb{P} (1_x, 1_y, 1_z, 1_t, 3_{\xi}, 3_{\zeta})$ defined by the equations
\begin{equation} \label{eq:nonrigidX}
\begin{cases}
\xi \zeta + x^3 g_3 + x^2 g_4 + x g_5 + g_6 = 0, \\
\xi + \zeta - 2 x^2 t = 0.
\end{cases}
\end{equation}
Let $\Phi \colon \mathcal{Y} \to \mathbb{P}$ be the weighted blow-up of $\mathbb{P}$ at $P = (1\!:\!0\!:\!\cdots\!:\!0)$ with weight $\operatorname{wt} (y, z, t, \xi, \zeta) = (1, 1, 1, 1, 2)$.
We set $Y := \Phi_*^{-1} X$ and $\varphi := \Phi|_Y \colon Y \to X$.
Let $\mathcal{E} \cong \mathbb{P} (1, 1, 1, 1, 2)$ be the exceptional divisor of $\Phi$ and let $E := \mathcal{E}|_Y$ be the $\varphi$-exceptional divisor.
Then we have isomorphisms
\begin{equation} \label{eq:nonrigidE}
E \cong (\xi \zeta - g_3 = \xi - 2 t = 0) \subset \mathbb{P} (1_y, 1_z, 1_t, 1_{\xi}, 2_{\zeta}) \cong \mathcal{E}.
\end{equation}
We have $K_Y = \varphi^*K_X + E$ and $\varphi$ is a divisorial contraction of discrepancy $1$.

Note that $E^3 = 3/2$.
We set $Q := (0\!:\!\cdots\!:\!0\!:\!1) \in E$, which is a $\frac{1}{2} (1, 1, 1)$ point of $Y$.
Let $\pi \colon X \dashrightarrow \mathbb{P}^2 = \mathbb{P} (1_y, 1_z, 1_t)$ be the projection to the coordinates $y, z, t$, and set $\rho := \pi \circ \varphi \colon Y \dashrightarrow \mathbb{P}^2$.

\begin{lemma} \label{lem:nonrigidnef}
The following assertions hold.
\begin{enumerate}
\item The rational map $\rho$ is defined outside the point $Q$.
\item The divisor $-K_Y$ is nef.
\end{enumerate}
\end{lemma}

\begin{proof}
We have
\[
\begin{split}
(y = z = t = 0) \cap X &= (y = z = t = \alpha \beta = \alpha + \beta = 0) \subset \mathbb{P} \\
&= \{P\},
\end{split}
\]
which shows that $\pi$ is defined outside $P$.
Hence the indeterminacy locus of $\rho$ is contained in $E$ and, under the isomorphism \eqref{eq:nonrigidE}, it coincides with the set
\[
\begin{split}
E \cap (y = z = t = 0) &= (y = z = t = \xi = 0) \subset \mathbb{P} (1_y, 1_z, 1_t, 1_{\xi}, 2_{\zeta}) \\
&= \{Q\}.
\end{split}
\]
This proves (1).

The sections $y, z, t$ lift to sections of $-K_Y \sim -\varphi^*K_X - E$, and hence $\left|-K_Y\right|$ is base point free away from $Q$.
This proves (2).
\end{proof}

Let $\psi \colon \hat{Y} \to Y$ be the Kawamata blow-up of $Y$ at $Q$.
Then $\psi$ resolves the indeterminacy of $\rho$ and $\hat{\rho} := \rho \circ \psi \colon \hat{Y} \to \mathbb{P}^2$ is everywhere defined.
Note that $\hat{Y}$ is a smooth projective variety.
The sections $y, z, t$ lift to sections of $-K_{\hat{Y}}$, and hence $\hat{\rho}$ is defined by the complete linear system $\left|-K_{\hat{Y}}\right|$.
It is clear that general fibers of $\pi$ are irreducible and reduced.
It follows that $\hat{\rho}$ has connected fibers, and hence it is an elliptic fibration.
Let $F \cong \mathbb{P}^2$ be the $\psi$-exceptional divisor and set $\hat{E} := \psi_*^{-1} E$.
We have $\psi^*E = \hat{E} + \frac{1}{2} F$.

\begin{lemma} \label{lem:nonrigidhat}
The following assertions hold.
\begin{enumerate}
\item The divisor $F$ is a section of $\hat{\rho}$ and $\hat{E}$ is a rational section of $\hat{\rho}$.
\item There are at most $3$ irreducible curves on $\hat{E}$ which are contracted by $\hat{\rho}$.
\item There is no prime divisor on $\hat{Y}$ that is mapped to a point by $\hat{\rho}$.
\end{enumerate}
\end{lemma}

\begin{proof}
A $\hat{\rho}$-fiber is rationally equivalent to $(-K_Y)^2$ and we have
\[
\begin{split}
((-K_Y)^2 \cdot F) &= ((- \psi^*K_Y - \frac{1}{2} F)^2 \cdot F) = \frac{1}{2^2} (E^3) = 1, \\
((-K_{\hat{Y}})^2 \cdot \hat{E}) &= ((-K_Y)^2 \cdot E) + \frac{1}{2^2} (F^2 \cdot \hat{E}) = (E^3) - \frac{1}{2^3} (F^3) = 1.
\end{split}
\]
Hence $\hat{E}$ is a rational $\hat{\rho}$-section and $F \cong \mathbb{P}^2$ is a $\hat{\rho}$-section.

Under the identification \eqref{eq:nonrigidE} of $E$, we further eliminate the variable $\xi = 2 t$, we have an isomorphism
\[
E \cong (2 t \zeta - g_3 = 0) \subset \mathbb{P} (1_y, 1_z, 1_t, 2_{\zeta}).
\]
Then the map $\rho|_E \colon E \dashrightarrow \mathbb{P}^2$ is the projection to the coordinates $y, z, t$, and the morphism $\psi|_{\hat{E}} \colon \hat{E} \to E$ is the weighted blow-up of $E$ at the point $Q = (0\!:\!\cdots\!:\!0\!:\!1)$ with weight $\operatorname{wt} (y, z, t) = \frac{1}{2} (1, 1, 1)$.
The irreducible curves on $\hat{E}$ that are contracted by $\hat{\rho}$  are exactly the proper transforms of the curves in $(t = g_3 = 0) \subset \mathbb{P} (1_y, 1_z, 1_t, 2_{\zeta})$.
There are at most $3$ irreducible curves in the above set since $g (y, z, 0) \ne 0$.
This proves (2).

Suppose that there is a prime divisor $G$ on $\hat{Y}$ that is mapped to a point by $\hat{\rho}$.
Then, $G \ne \hat{E}, F$ by (1), and we see that $D := \varphi_*\psi_*G$ is a prime divisor on $X$ that is mapped to a point by $\pi$.
This is clearly impossible since a fiber of $\pi$ is a complete intersection of two distinct divisors in $|\mathfrak{m}_P (A)|$ and cannot be a divisor on $X$.
This proves (3).
\end{proof}

Let $\hat{\chi} \colon \hat{Y} \dashrightarrow \hat{Y}$ be the reflection of the elliptic fibration with respect to the section $F$.
The birational involution $\hat{\chi}$ is an isomorphism in codimension $1$ since $K_Y$ is $\hat{\rho}$-nef, and it descends to a birational involution $\chi \colon Y \dashrightarrow Y$ which is an isomorphism in codimension $1$ since $\hat{\chi}_* F = F$.
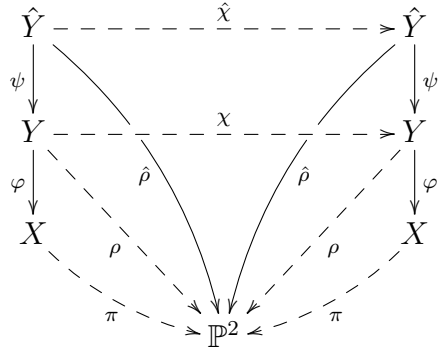
\[
\xymatrix{
\ar|(.39)\hole@/^13pt/[rrddd]_{\hat{\rho}} \hat{Y} \ar[d]_{\psi} \ar@{-->}[rrrr]^{\hat{\chi}} & & & & \hat{Y} \ar[d]^{\psi} \ar|(.39)\hole@/_13pt/[llddd]^{\hat{\rho}}\\
Y \ar[d]_{\varphi} \ar@{-->}[rrrr]^{\chi} \ar@{-->}[rrdd]_{\rho} & & & & \ar@{-->}[lldd]^{\rho} Y \ar[d]^{\varphi} \\
X \ar@/_7pt/@{-->}[rrd]_{\pi} & & & & \ar@/^7pt/@{-->}[lld]^{\pi} X \\
& & \mathbb{P}^2 & &}
\]

\begin{lemma}
\label{lem:SDSexcldivcont}
If there are infinitely many reducible $\hat{\rho}$-fibers, then $\varphi$ is not a maximal extraction.
\end{lemma}

\begin{proof}
Let $\Gamma_{\lambda}$, $\lambda \in \Lambda$, be reducible $\hat{\rho}$-fibers.
Possibly removing finitely many elements from $\Lambda$, we may assume that, for any $\lambda \in \Lambda$, $\Gamma_{\lambda}$ does not contain any irreducible curve on $\hat{E}$ contracted by $\hat{\rho}$ (cf. (2) of Lemma~\ref{lem:nonrigidhat}).

We think of $\Gamma_{\lambda}$ as the $1$-cycle $\hat{\rho}^*\ell_1 \cdot \hat{\rho}^* \ell_2$, where $\ell_1$ and $\ell_2$ are lines on $\mathbb{P}^2$ such that $\ell_1 \cap \ell_2 = \{\hat{\rho} (\Gamma)\}$.
Note that $(-K_Y \cdot \Gamma_{\lambda}) = 0$ and $\Gamma_{\lambda}$ is an effective $1$-cycle with integral coefficients.
Then we can write $\Gamma_{\lambda} = \Delta_{\lambda} + \Xi_{\lambda}$, where $\Delta_{\lambda}$ is an irreducible curve such that $(F \cdot \Delta_{\lambda}) = 1$ and $\Xi$ is an effective $1$-cycle with integral coefficients such that $\operatorname{Supp} (\Xi_{\lambda}) \cap F = \emptyset$.
We have $(-K_{\hat{Y}} \cdot (\Delta_{\lambda} + \Xi_{\lambda})) = 0$ and $-K_{\hat{Y}} = - \psi^*K_Y - \frac{1}{2} E$.
It follows that
\[
(-K_Y \cdot \psi_*\Delta_{\lambda}) + (-K_Y \cdot \psi_*\Xi_{\lambda}) = \frac{1}{2}.
\]
The effective $1$-cycle $\psi_*\Xi_{\lambda}$ is integral and its support does not pass through the unique non-Gorenstein singular point $Q$ of $Y$, and hence $(-K_Y \cdot \psi_*\Xi_{\lambda}) \in \mathbb{Z}$.
It follows that $(-K_Y \cdot \psi_*\Delta) = 1/2$ and $(-K_Y \cdot \psi_*\Xi) = 0$ since $-K_Y$ is nef by Lemma~\ref{lem:nonrigidnef}.
By our choice of $\Gamma_{\lambda}$, no component of $\psi_*\Xi_{\lambda} \ne 0$ is contained in $E$.
Hence $(E \cdot \psi_*\Xi) > 0$ since $\varphi_*\psi_*\Xi$ is a nonzero effective $1$-cycle on $X$ and we have
\[
0 = (-K_Y \cdot \psi_*\Xi) = (-K_X \cdot \varphi_*\psi_*\Xi) - (E \cdot \psi_*\Xi).
\]
The intersection number of $-K_Y$ and any irreducible component of $\psi_*\Xi$ is $0$ since $(-K_Y \cdot \psi_*\Xi_{\lambda}) = 0$ and $\psi_*\Xi_{\lambda}$ is effective.
Thus, there is an irreducible component of $\psi_*\Xi_{\lambda}$, denoted by $\Theta_{\lambda}$, such that $(-K_Y \cdot \Theta_{\lambda}) = 0$ and $(E \cdot \Theta_{\lambda}) > 0$.

It follows that there are infinitely many irreducible curves $\{\Theta_{\lambda}\}$ on $Y$ that intersect $K_Y$ trivially and $E$ positively.
By \cite[Lemma 2.20]{Oka18}, $\varphi$ is not a maximal extraction.
\end{proof}

\begin{lemma} \label{lem:nonrigidlink}
If there are at most finitely many reducible $\hat{\rho}$-fibers, then $\chi$ is not an isomorphism.
In particular, the induced birational map $\sigma := \varphi^{-1} \chi \circ \varphi \colon X \dashrightarrow X$ is an elementary self-link.
\end{lemma}

\begin{proof}
Let $G$ be a prime divisor on $\hat{Y}$ which is $\hat{\rho}$-vertical, that is, $\hat{\rho} (G)$ is a proper closed subvariety of $\mathbb{P}^2$.
By Lemma~\ref{lem:nonrigidhat}, $C := \hat{\rho} (G)$ is an irreducible curve.
If $\hat{\rho}^*C$ is reducible, then it contains a prime divisor $G' \ne G$ in its support and we also have $\hat{\rho} (G') = C$.
This implies that the $\hat{\rho}$-fibers over any point of $C$ are reducible.
This is impossible by assumption.
It follows that any prime vertical divisor on $\hat{Y}$ is the pullback of an irreducible curve on $\mathbb{P}^2$.
In particular, the sum of vertical divisors on $\hat{Y}$ is linearly equivalent to a multiple of $-K_{\hat{Y}}$.

Now suppose that $\chi$ is an isomorphism.
Then $\chi_* E = E$ and, on the generic fiber $\hat{Y}_{\eta}$ of $\hat{\rho}$, we see that $(\hat{E} - F)|_{\hat{Y}_{\eta}} \in \operatorname{Pic}^0 (\hat{Y}_{\eta})$ is a $2$-torsion.
It follows that $\hat{E} - F$ is $\mathbb{Q}$-linearly equivalent to a sum of $\hat{\rho}$-vertical divisors.
By what we have proved, we have $\hat{E} - F \sim_{\mathbb{Q}} m K_{\hat{Y}}$ for some $m \in \mathbb{Q}$.
This is impossible since $\hat{E}, F$ and $K_{\hat{Y}}$ generate $\operatorname{Pic} (\hat{Y})_{\mathbb{Q}} \cong \mathbb{Q}^3$.
Therefore, $\chi$ is not an isomorphism.
\end{proof}

\begin{proof}[Proof of Proposition~\ref{prop:DSbirinv}]
Let $X$ be a factorial sextic double solid and let $P \in X$ be an isolated $cA_2$ point.
We identify $X$ with the complete intersection in $\mathbb{P}$ defined by the equations \eqref{eq:nonrigidX}.

Let $\varphi \colon Y \to X$ be the divisorial contraction of discrepancy $1$ constructed as above.
Interchanging the role of $\zeta$ and $\xi$, we obtain another divisorial contraction $\varphi' \colon Y' \to X$ of discrepancy $1$ centered at $P$.
These are all the divisorial contractions of discrepancy $1$ centered at $P$.
Thus, the assertion follows from Lemmas \ref{lem:SDSexcldivcont} and \ref{lem:nonrigidlink}.
\end{proof}

\subsection{Birational rigidity of sextic double solids}

\begin{proposition}
\label{prop:SDSexclexc}
    Let $X$ be a factorial sextic double solid
    admitting an ordinary $cA_2$ point $P$.
    Then a divisorial contraction $\varphi \colon Y \to X$ of exceptional type with center $P$ is not a maximal extraction.
\end{proposition}

\begin{proof}
Let $\varphi \colon Y \to X$ be a divisorial contraction of exceptional type with center a $P$ which is the weighted blow-up with weights $4, 3, 2, 1$ after an appropriate local analytic identification of the germ $P \in X$ by Theorem~\ref{thm:cldccA2}.
Suppose that it is a maximal extraction.
Then there is a mobile linear system $\mathcal{M} \sim_{\mathbb{Q}} - n K_X$ such that $(X, \frac{1}{n} \mathcal{M})$ is not canonical at the exceptional divisor $E$ of $\varphi$.
Let $D_1, D_2 \in \mathcal M$ be general members.
We can choose a general member $H \in |\mathcal O_X (1)|$ passing through $P$ such that $H$ does not contain any component of the curve $D_1 \cap D_2$.
By Theorem~\ref{thm:locineqwblcA} and Lemma~\ref{lem:intom}, we have
\[
2 n^2 = (H \cdot D_1 \cdot D_2) \ge \operatorname{mult}_P (D_1 \cdot D_2) > \frac{9}{4} n^2.
\]
This is a contradiction and $\varphi$ is not a maximal extraction.
\end{proof}

\begin{theorem} \label{thm:sextic-double-solids}
    Let $X$ be a factorial sextic double solid with only isolated $cA_1$ points and ordinary $cA_2$ points.
    Then $X$ is birationally rigid.
\end{theorem}

\begin{proof}
By \cite[Lemma 3.1]{CP10} and \cite[Propositions 4.8 and 4.9]{KOPP24}, no smooth point, no curve, and no $cA_1$ point on $X$ are maximal centers.
Therefore, by Theorem~\ref{thm:chractBR}, the assertion follows from Propositions \ref{prop:DSbirinv} and \ref{prop:SDSexclexc}.
\end{proof}

\section{Fano 3-fold WCIs of index 1}

We investigate the birational rigidity of prime Fano 3-folds embedded as a complete intersection in a weighted projective space that admit $cA$ points.
By a \textit{prime Fano $3$-fold}, we mean a normal projective $\mathbb{Q}$-factorial $3$-fold with only terminal singularities such that the anticanonical divisor is ample and the class group is isomorphic to $\mathbb{Z}$.
The \textit{index} of a prime Fano $3$-fold $X$ is defined to be the positive integer $\iota_X$ such that $-K_X \sim \iota_X A$, where $A \in \operatorname{Cl} (X)$ is the ample generator of $\operatorname{Cl} (X) \cong \mathbb{Z}$.

\begin{setting} \label{set:brWCI}
Let $X$ be a prime Fano $3$-fold which is either a weighted hypersurface of index $1$ listed in Table \ref{table:Fanohyp} or a weighted complete intersection of codimension $2$ and index $1$ listed in Table \ref{table:FanoWCI}.
Let $\mathbb{P}$ be the ambient weighted projective space of $X$.
We assume that $X$ is quasismooth along $X \cap \operatorname{Sing} (\mathbb{P})$, where $\operatorname{Sing} (\mathbb{P})$ is the singular locus of $\mathbb{P}$.
\end{setting}

\begin{theorem} \label{thm:brWCI}
Let $X$ be as in Setting \ref{set:brWCI} and let $k_{\mathrm{cA}}$ be the positive integer given in the 5th and 4th columns of Tables \ref{table:Fanohyp} and \ref{table:FanoWCI}, respectively.
Suppose that $X$ has only $cA_k$ points, where $k \le k_{\mathrm{cA}}$, in addition to terminal cyclic quotient singularities.
Then $X$ is birationally rigid.
\end{theorem}

\subsection{Basics on weighted projective varieties}
\label{sec:WPV}

We give basic definitions concerning weighted projective spaces and its closed subvarieties.
Let $\mathbb{P} = \mathbb{P} (a_0, \dots, a_N)$ be a weighted projective space and let $x_0, \dots, x_N$ be its homogeneous coordinates of weights $a_0, \dots, a_N$, respectively.
We say that $\mathbb{P}$ is \textit{well-formed} if the greatest common divisor of any $N$ of the weights $a_0, \dots, a_N$ is $1$.
Let $X$ be a closed subvariety of $\mathbb{P} = \mathbb{P} (a_0, \dots, a_N)$.
We say that $X$ is \textit{well-formed} if $\mathbb{P}$ is well-formed and
\[
\operatorname{codim}_X (X \cap \operatorname{Sing} (\mathbb{P})) \ge 2.
\]
If $X \subset \mathbb{P}$ is normal and well-formed, then the adjunction holds, that is, we have $(K_{\mathbb{P}} + X)|_X = K_X$.
We define
\[
C_X := \operatorname{Spec} (\mathbb{C} [x_0, \dots, x_N]/I) \subset \mathbb{A}^{N+1},
\]
where $I$ is the defining homogeneous ideal of $X$ in $\mathbb{P}$, and let $\pi \colon C_X^* \to X$ be the natural projection, where $C_X^* := C_X \setminus \{o\}$.
We say that $X$ is \textit{quasismooth} if $C_X^*$ is smooth.
For a subset $S \subset X$, we say that $X$ is \textit{quasismooth along $S$} if $C_X^*$ is smooth along $\pi^{-1} (S)$.

For a homogeneous coordinate $s \in \{x_0, \dots, x_N\}$, we denote by $P_s$ the point $(0\!:\!\cdots\!:\!1\!:\!\cdots\!:\!0) \in \mathbb{P}$ at which only the coordinate $s$ does not vanish.
For homogeneous polynomials $f_1, \dots, f_m \in \mathbb{C} [x_0, \dots, x_N]$, we define
\[
(f_1 = \cdots = f_m = 0)
\]
to be the closed subscheme of $\mathbb{P}$ defined by the homogeneous ideal $(f_1, \dots, f_m)$ of the graded ring $\mathbb{C} [x_0, \ldots, x_N]$.

\begin{definition}
Let $P \in X$ be a point.
We say that a set $\{g_1, \dots, g_m\}$ of homogeneous polynomials $g_1, \dots, g_m \in \mathbb{C} [x_0, \dots, x_N]$ \textit{isolates} $P$ or is a $P$-\textit{isolating set} if $P$ is an isolated component of the set
\[
(g_1 = \cdots = g_m = 0) \cap X.
\]
\end{definition}

\begin{definition}
Let $P \in X$ be a point that is contained in the smooth locus of $\mathbb{P}$ and let $L$ be a Cartier divisor class on $X$.
For positive integers $k$ and $l$, we define $|\mathfrak{m}^k (lL)|$ to be the sublinear system of $|lL|$ consisting of divisors vanishing at $P$ with order at least $k$.
We say that $L$ \textit{isolates} $P$ or is a $P$-\textit{isolating class} if $P$ is an isolated component of the base locus of $|\mathfrak{m}_P^k (k L)|$ for some positive integer $k$.
\end{definition}

\begin{lemma} \label{lem:isol}
Let $P \in X$ be a point that is contained in the smooth locus of $\mathbb{P}$.
If a set $\{g_1, \dots, g_m\}$ of homogeneous polynomials $g_1, \dots, g_m \in \mathbb{C} [x_0, \dots, x_N]$ isolates $P \in X$, then $l A$ isolates $P$, where
\[
l = \max \{\, \deg g_i \mid 1 \le i \le m \,\}.
\]
\end{lemma}

\begin{proof}
We set $L := l A \in \operatorname{Cl} (X)$.
We define $n := \operatorname{lcm} \{\, \deg g_i \mid 1 \le i \le m \,\}$, $k_i := n/\deg g_i$ for $1 \le i \le m$ and $k := n/l$.
Then we have $k L = n A$ and $k_i \ge k$ for $1 \le i \le m$.
We have $\operatorname{ord}_P (D_i) \ge 1$ for the effective divisor
\[
D_i := (g_i = 0) \cap X \sim (\deg g_i)A = \frac{n}{k_i} A.
\]
It follows that $k_i D_i \in |\mathfrak{m}_P^k (kL)|$ since $\operatorname{ord}_P (k_i D_i) \ge k_i \ge k$.
The base locus of $|\mathfrak{m}_P^k (kL)|$ is contained in the set $D_1 \cap \cdots \cap D_m$, hence $P$ is its isolated component.
This shows that $L = l A$ isolates $P$.
\end{proof}

\begin{lemma} \label{lem:isoldivs}
Let $V$ be a normal complete intersection of codimension $c \in \{1, 2\}$ in $\mathbb{P} (a_0, \dots, a_N)$ and let $A$ be the Weil divisor class of $X$ corresponding to $\mathcal{O}_V (1)$.
\begin{enumerate}
\item Let $P \in V$ be a point.
\begin{enumerate}
\item There exists a $P$-isolating class $lA$ for some
\[
l \le \max \{\, \operatorname{lcm} (a_i, a_j) \mid i, j \in \{0, 1, \dots, N\} \, \}.
\]
\item If there is $m$ such that a power of $x_m$ appears in one of the defining polynomials of $V$ with nonzero coefficient, then there exist a $P$-isolating class $lA$ for some
\[
l \le \max \{\, \operatorname{lcm} (a_i, a_j) \mid i, j \in \{0, 1, \dots, N\} \setminus \{m\} \, \}.
\]
\item If $c = 2$ and there are distinct $m_1, m_2$ such that
\begin{equation} \label{eq:isoldivsassump}
\bigcap_{i \in \{0, \dots, N\} \setminus \{m_1, m_2\}} (x_i = 0) \cap V = \emptyset,
\end{equation}
then there exists a $P$-isolating class $l A$ for some
\[
l \le \max \{\, \operatorname{lcm} (a_i, a_j) \mid i, j \in \{0, 1, \dots, N\} \setminus \{m_1, m_2\} \,\}.
\]
\end{enumerate}
\item Let $P \in U_{x_r} = (x_n \ne 0) \cap V$ be a point, where $r \in \{0, \dots, N\}$.
\begin{enumerate}
\item There exists a $P$-isolating class $l A$ for some
\[
l \le \max \{\, \operatorname{lcm} (a_r, a_j) \mid j \in \{0, 1, \dots, N\} \,\}.
\]
\item If there is $m \ne r$ such that a power of $x_m$ appears in one of the defining polynomials of $V$ with nonzero coefficient, then there exists a $P$-isolating class $l A$ for some
\[
l \le \max \{\, \operatorname{lcm} (a_r, a_j) \mid j \in \{0, 1, \dots, N\} \setminus \{m\} \, \}.
\]
\item If $c = 2$ and there are distinct $m_1, m_2$ such that $m_i \ne r$ for $i = 1, 2$ and
\[
\bigcap_{i \in \{0, \dots, N\} \setminus \{m_1, m_2\}} (x_i = 0) \cap V = \emptyset,
\]
then there exists a $P$-isolating class $lA$ for some
\[
l \le \max \{\, \operatorname{lcm} (a_r, a_j) \mid j \in \{0, 1, \dots, N\} \setminus \{m_1, m_2\} \,\}.
\]
\end{enumerate}
\end{enumerate}
\end{lemma}

\begin{proof}
We write $P = (\alpha_0\!:\!\cdots\!:\!\alpha_N)$, where $\alpha_i \in \mathbb{C}$.
For $i \ne j$, we set $b_{ij} = \operatorname{lcm} (a_i, a_j)$ and
\[
g_{ij} := \alpha_j^{b_{ij}/a_j} x_i^{b_{ij}/a_i} - \alpha_i^{b_{ij}/a_i} x_j^{b_{ij}/a_j},
\]
which is a homogeneous polynomial of degree $b_{ij} = \operatorname{lcm} (a_i, a_j)$.
The set of polynomials
\[
\{\, g_{ij} \mid i, j \in \{0, 1, \dots, N\} \}
\]
clearly isolates $P$, which shows (1-a).
If $P \in U_r = (x_r \ne 0) \cap V$, then the set
\[
\{\, g_{rj} \mid j \in \{0, 1, \dots, N\} \}
\]
isolates $P$.

Suppose that a power of $x_m$ appears in one of the defining polynomials of $V$.
We consider the projection
\[
\pi \colon V \dashrightarrow \mathbb{P} (a_0, \dots, a_{m-1},a_{m+1},\dots, a_N).
\]
Any fiber of $\pi$ is a finite set of points.
Then the set of polynomials
\[
\{\, g_{i j} \mid i, j \in \{0, \dots, N\} \setminus \{m\}\}
\]
isolates the point $\pi (P)$.
This shows that the above set is a $P$-isolating class as well since $\pi^{-1} (\pi (P))$ is a finite set.
This proves (1-b).
If in addition $P \in U_r$ and $m \ne r$, then $\pi (P) \in (x_r \ne 0) \cap \pi (X)$.
Since the set
\[
\{\, g_{rj} \mid j \in \{0, 1, \dots, N\} \setminus \{m\} \}
\]
isolates $\pi (P)$, the same set isolates $P$ as well.
This shows (2-b).

Suppose that the assumption of (1-c) is satisfied.
We consider the projection
\[
\pi \colon V \to \mathbb{P} (a_0, \dots, a_{m_1-1},a_{m_1+1},\dots,a_{m_2-1},a_{m_2+1},\dots,a_N) =: \mathbb{P}',
\]
which is indeed everywhere defined by the assumption \eqref{eq:isoldivsassump}.
Let $A' \in \operatorname{Cl} (\mathbb{P}')$ be the Weil divisor class corresponding to $\mathcal{O}_{\mathbb{P}'} (1)$.
We have $\pi^*A' = A$ and this implies that $\pi$ does not contract any curve.
The set
\[
\{\, g_{ij} \mid i, j \in \{0, \dots, N\} \setminus \{m_1, m_2\}\}
\]
isolates the point $\pi (P) \in \mathbb{P}'$ and hence the same set isolates $P$ since $\pi^{-1} (\pi (P))$ is a finite set.
This shows (1-c).
If in addition $P \in U_r$ for some $r \ne m_1, m_2$, then $\pi (P) \in (x_r \ne 0) \cap \pi (V)$.
In this case the set
\[
\{\, g_{rj} \mid \{j \in \{0, \dots, N\} \setminus \{m_1, m_2\}\}
\]
isolates the point $\pi (P) \in \mathbb{P}'$ and the same isolates $P$.
This shows (2-c) and the proof is complete.
\end{proof}

\subsection{Proof of Theorem~\ref{thm:brWCI}}

\subsubsection{Weighted hypersurfaces}

Let $X = X_d \subset \mathbb{P} (a_0, \dots, a_4) =: \mathbb{P}$ be a member of family listed in Table \ref{table:Fanohyp} satisfying the assumption of Setting \ref{set:brWCI}.
We assume that $a_0 \le \cdots \le a_4$ and set $A := -K_X \in \operatorname{Cl} (X)$ which is the Weil divisor class corresponding to $\mathcal{O}_X (1)$.
Let $x, y, z, t, w$ be the homogeneous coordinates of $\mathbb{P}$ of weights $a_0 = 1, a_1, a_2, a_3, a_4$, respectively, and we denote by $F = F (x, y, z, t, w)$ the homogeneous polynomial of degree $d$ that defines $X$ in $\mathbb{P}$.

\begingroup
\renewcommand{\arraystretch}{1.1}
\begin{longtable}{clccccl}
\caption{Fano $3$-fold weighted hypersurfaces of index $1$}
\label{table:Fanohyp}
\\
\hline\hline
\textnumero & $X_d \subset \mathbb{P} (a_0,\dots,a_4)$ & $-K_X^3$ & $l_{\mathrm{ic}}$ & $k_{\mathrm{cA}}$ & Case & $F (0,0,z,t,w)$ \\
\hline\hline
\endfirsthead
\hline\hline
\textnumero & $X_d \subset \mathbb{P} (a_0,\dots,a_4)$ & $-K_X^3$ & $l_{\mathrm{ic}}$ & $k_{\mathrm{cA}}$ & Case & $F (0,0,z,t,w)$ \\
\hline\hline
\endhead
\rowcolor{lightgray}
6 & $X_{8} \subset \mathbb{P} (1,1,1,2,4)$ & $1$ &  & $1$ & $\bigstar$ & \\
10 & $X_{10} \subset \mathbb{P} (1,1,1,3,5)$ & $2/3$ &  & $1$ & $\bigstar$ & \\
\rowcolor{lightgray}
11 & $X_{10} \subset \mathbb{P} (1,1,2,2,5)$ & $1/2$ & $2$ & $3$ & $\clubsuit$ & \\
14 & $X_{12} \subset \mathbb{P} (1,1,1,4,6)$ & $1/2$ &  & $1$ & $\bigstar$ & \\
\rowcolor{lightgray}
15 & $X_{12} \subset \mathbb{P} (1,1,2,3,6)$ & $1/3$ &  & $1$ & $\bigstar$ & \\
17 & $X_{12} \subset \mathbb{P} (1,1,3,4,4)$ & $1/4$ & $4$ & $3$ & $\heartsuit$ & $z^4 + \phi_{(17)} (t, w)$ \\
\rowcolor{lightgray}
18 & $X_{12} \subset \mathbb{P} (1,2,2,3,5)$ & $1/5$ &  & $1$ & $\bigstar$ & \\
19 & $X_{12} \subset \mathbb{P} (1,2,3,3,4)$ & $1/6$ & $6$ & $3$ & $\heartsuit$ & $w^3 + \phi_{(19)} (z, t)$ \\
\rowcolor{lightgray}
21 & $X_{14} \subset \mathbb{P} (1,1,2,4,7)$ & $1/4$ & $4$ & $3$ & $\heartsuit$ & $w^2 + z \phi_{(21)} (z, t)$ \\
22 & $X_{14} \subset \mathbb{P} (1,2,2,3,7)$ & $1/6$ & $6$ & $3$ & $\clubsuit$ & \\
\rowcolor{lightgray}
25 & $X_{15} \subset \mathbb{P} (1,1,3,4,7)$ & $5/28$ & $7$ & $2$ & $\heartsuit$ & $\underline{w t^2} + \underline{t^3 z} + z^5$ \\
26 & $X_{15} \subset \mathbb{P} (1,1,3,5,6)$ & $1/6$ & $6$ & $3$ & $\heartsuit$ & $w^2 z + t^3 + z^5$ \\
\rowcolor{lightgray}
27 & $X_{15} \subset \mathbb{P} (1,2,3,5,5)$ & $1/10$ & $10$ & $3$ & $\heartsuit$ & $\phi_{(27)} (t, w) + z^5$ \\
28 & $X_{15} \subset \mathbb{P} (1,3,3,4,5)$ & $1/12$ & $12$ & $3$ & $\clubsuit$ & \\
\rowcolor{lightgray}
29 & $X_{16} \subset \mathbb{P} (1,1,2,5,8)$ & $1/5$ & $5$ & $3$ & $\heartsuit$ & $w^2 + (t^2 z^3) + z^8$ \\
30 & $X_{16} \subset \mathbb{P} (1,1,3,4,8)$ & $1/6$ & $4$ & $5$ & $\heartsuit$ & $w^2 + t^4 + (t z^4)$ \\
\rowcolor{lightgray}
31 & $X_{16} \subset \mathbb{P} (1,1,4,5,6)$ & $2/15$ & $6$ & $4$ & $\heartsuit$ & $w^2 z + (w t^2) + z^4$ \\
32 & $X_{16} \subset \mathbb{P} (1,2,3,4,7)$ & $2/21$ & $14$ & $2$ & $\heartsuit$ & $(w z^3) + t^4 + (t z^4)$ \\
\rowcolor{lightgray}
34 & $X_{18} \subset \mathbb{P} (1,1,2,6,9)$ & $1/6$ & $6$ & $3$ & $\heartsuit$ & $w^2 + c (z, t)$ \\
35 & $X_{18} \subset \mathbb{P} (1,1,3,5,9)$ & $2/15$ & $5$ & $5$ & $\heartsuit$ & $w^2 + t^3 z + z^6$ \\
\rowcolor{lightgray}
36 & $X_{18} \subset \mathbb{P} (1,1,4,6,7)$ & $3/28$ & $7$ & $4$ & $\heartsuit$ & $w^2 z + t^3 + t z^3$ \\
37 & $X_{18} \subset \mathbb{P} (1,2,3,4,9)$ & $1/12$ & $6$ & $7$ & $\heartsuit$ & $w^2 + (t^3 z^2) + z^6$ \\
\rowcolor{lightgray}
38 & $X_{18} \subset \mathbb{P} (1,2,3,5,8)$ & $3/40$ & $10$ & $4$ & $\heartsuit$ & $\underline{w t^2} + \underline{t^3 z} + z^6$ \\
41 & $X_{20} \subset \mathbb{P} (1,1,4,5,10)$ & $1/10$ & $5$ & $7$ & $\heartsuit$ & $w^2 + t^4 + z^5$ \\
\rowcolor{lightgray}
42 & $X_{20} \subset \mathbb{P} (1,2,3,5,10)$ & $1/15$ & $10$ & $5$ & $\heartsuit$ & $w^2 + t^4 + (t z^5)$ \\
43 & $X_{20} \subset \mathbb{P} (1,2,4,5,9)$ & $1/18$ & $18$ & $3$ & $\heartsuit$ & $t^4 + z^5$ \\
\rowcolor{lightgray}
44 & $X_{20} \subset \mathbb{P} (1,2,5,6,7)$ & $1/21$ & $14$ & $5$ & $\heartsuit$ & $w^2 t + z^4$ \\
45 & $X_{20} \subset \mathbb{P} (1,3,4,5,8)$ & $1/24$ & $24$ & $3$ & $\heartsuit$ & $w^2 z + t^4 + z^5$ \\
\rowcolor{lightgray}
46 & $X_{21} \subset \mathbb{P} (1,1,3,7,10)$ & $1/10$ & $10$ & $3$ & $\heartsuit$ & $t^3 + z^7$ \\
47 & $X_{21} \subset \mathbb{P} (1,1,5,7,8)$ & $3/40$ & $8$ & $5$ & $\heartsuit$ & $w^2 z + t^3$ \\
\rowcolor{lightgray}
48 & $X_{21} \subset \mathbb{P} (1,2,3,7,9)$ & $1/18$ & $18$ & $3$ & $\heartsuit$ & $w^2 z + t^3 + z^7$ \\
49 & $X_{21} \subset \mathbb{P} (1,3,5,6,7)$ & $1/30$ & $15$ & $7$ & $\heartsuit$ & $w^3 + (t z^3)$ \\
\rowcolor{lightgray}
51 & $X_{22} \subset \mathbb{P} (1,1,4,6,11)$ & $1/12$ & $6$ & $7$ & $\heartsuit$ & $w^2 + t^3 z + z^4 t$ \\
52 & $X_{22} \subset \mathbb{P} (1,2,4,5,11)$ & $1/20$ & $10$ & $7$ & $\spadesuit$ & $w^2 + (t^2 z^3)$ \\
\rowcolor{lightgray}
53 & $X_{24} \subset \mathbb{P} (1,1,3,8,12)$ & $1/12$ & $8$ & $5$ & $\heartsuit$ & $w^2 + t^3 + z^8$ \\
54 & $X_{24} \subset \mathbb{P} (1,1,6,8,9)$ & $1/18$ & $9$ & $7$ & $\heartsuit$ & $w^2 z + t^3 + z^4$ \\
\rowcolor{lightgray}
55 & $X_{24} \subset \mathbb{P} (1,2,3,7,12)$ & $1/21$ & $14$ & $5$ & $\heartsuit$ & $w^2 + t^3 z + z^8$ \\
56 & $X_{24} \subset \mathbb{P} (1,2,3,8,11)$ & $1/22$ & $22$ & $3$ & $\heartsuit$ & $t^3 + z^8$ \\
\rowcolor{lightgray}
57 & $X_{24} \subset \mathbb{P} (1,3,4,5,12)$ & $1/30$ & $15$ & $7$ & $\heartsuit$ & $w^2 + t^4 z + z^6$ \\
58 & $X_{24} \subset \mathbb{P} (1,3,4,7,10)$ & $1/35$ & $30$ & $3$ & $\heartsuit$ & $w^2 z + (w t^2) + z^6$ \\
\rowcolor{lightgray}
59 & $X_{24} \subset \mathbb{P} (1,3,6,7,8)$ & $1/42$ & $21$ & $7$ & $\heartsuit$ & $w^3 + z^4$ \\
60 & $X_{24} \subset \mathbb{P} (1,4,5,6,9)$ & $1/45$ & $36$ & $4$ & $\heartsuit$ & $w^2 t + (w z^3) + t^4$ \\
\rowcolor{lightgray}
61 & $X_{25} \subset \mathbb{P} (1,4,5,7,9)$ & $5/252$ & $36$ & $4$ & $\heartsuit$ & $w^2 t + z^5$ \\
62 & $X_{26} \subset \mathbb{P} (1,1,5,7,13)$ & $2/35$ & $7$ & $9$ & $\heartsuit$ & $w^2 + t^3 z$ \\
\rowcolor{lightgray}
63 & $X_{26} \subset \mathbb{P} (1,2,3,8,13)$ & $1/24$ & $8$ & $11$ & $\spadesuit$ & $w^2 + (t z^6)$ \\
64 & $X_{26} \subset \mathbb{P} (1,2,5,6,13)$ & $1/30$ & $10$ & $11$ & $\heartsuit$ & $w^2 + (t z^4)$ \\
\rowcolor{lightgray}
65 & $X_{27} \subset \mathbb{P} (1,2,5,9,11)$ & $3/110$ & $22$ & $5$ & $\heartsuit$ & $w^2 z + t^3$ \\
66 & $X_{27} \subset \mathbb{P} (1,5,6,7,9)$ & $1/70$ & $35$ & $7$ & $\heartsuit$ & $w^3 + w z^3 + t^3 z$ \\
\rowcolor{lightgray}
67 & $X_{28} \subset \mathbb{P} (1,1,4,9,14)$ & $1/18$ & $9$ & $7$ & $\heartsuit$ & $w^2 + z^7$ \\
68 & $X_{28} \subset \mathbb{P} (1,3,4,7,14)$ & $1/42$ & $21$ & $7$ & $\heartsuit$ & $w^2 + t^4 + z^7$ \\
\rowcolor{lightgray}
69 & $X_{28} \subset \mathbb{P} (1,4,6,7,11)$ & $1/66$ & $44$ & $5$ & $\heartsuit$ & $w^2 z + t^4$ \\
70 & $X_{30} \subset \mathbb{P} (1,1,4,10,15)$ & $1/20$ & $10$ & $7$ & $\heartsuit$ & $w^2 + t^3 + t z^5$ \\
\rowcolor{lightgray}
71 & $X_{30} \subset \mathbb{P} (1,1,6,8,15)$ & $1/24$ & $8$ & $11$ & $\heartsuit$ & $w^2 + t^3 z + z^5$ \\
72 & $X_{30} \subset \mathbb{P} (1,2,3,10,15)$ & $1/30$ & $10$ & $11$ & $\heartsuit$ & $w^2 + t^3 + z^{10}$ \\
\rowcolor{lightgray}
73 & $X_{30} \subset \mathbb{P} (1,2,6,7,15)$ & $1/42$ & $14$ & $11$ & $\heartsuit$ & $w^2 + z^5$ \\
74 & $X_{30} \subset \mathbb{P} (1,3,4,10,13)$ & $1/52$ & $39$ & $4$ & $\heartsuit$ & $w^2 z + t^5 + t z^5$ \\
\rowcolor{lightgray}
75 & $X_{30} \subset \mathbb{P} (1,4,5,6,15)$ & $1/60$ & $20$ & $11$ & $\heartsuit$ & $w^2 + t^5 + z^6$ \\
76 & $X_{30} \subset \mathbb{P} (1,5,6,8,11)$ & $1/88$ & $55$ & $5$ & $\heartsuit$ & $w^2 t + t^3 z + z^5$ \\
\rowcolor{lightgray}
77 & $X_{32} \subset \mathbb{P} (1,2,5,9,16)$ & $1/45$ & $18$ & $9$ & $\heartsuit$ & $w^2 + t^3 z$ \\
78 & $X_{32} \subset \mathbb{P} (1,4,5,7,16)$ & $1/70$ & $28$ & $9$ & $\heartsuit$ & $w^2 + z^5 t$ \\
\rowcolor{lightgray}
79 & $X_{33} \subset \mathbb{P} (1,3,5,11,14)$ & $1/70$ & $42$ & $5$ & $\heartsuit$ & $w^2 z + t^3$ \\
80 & $X_{34} \subset \mathbb{P} (1,3,4,10,17)$ & $1/60$ & $30$ & $7$ & $\heartsuit$ & $w^2 + t^3 z + z^6 t$ \\
\rowcolor{lightgray}
81 & $X_{34} \subset \mathbb{P} (1,4,6,7,17)$ & $1/84$ & $28$ & $11$ & $\heartsuit$ & $w^2 + t^4 z$ \\
82 & $X_{36} \subset \mathbb{P} (1,1,5,12,18)$ & $1/30$ & $12$ & $9$ & $\heartsuit$ & $w^2 + t^3$ \\
\rowcolor{lightgray}
83 & $X_{36} \subset \mathbb{P} (1,3,4,11,18)$ & $1/66$ & $33$ & $7$ & $\heartsuit$ & $w^2 + z^9$ \\
84 & $X_{36} \subset \mathbb{P} (1,7,8,9,12)$ & $1/168$ & $63$ & $9$ & $\heartsuit$ & $w^3 + w z^3 + t^4$ \\
\rowcolor{lightgray}
85 & $X_{38} \subset \mathbb{P} (1,3,5,11,19)$ & $2/165$ & $33$ & $9$ & $\heartsuit$ & $w^2 + t^3 z$ \\
86 & $X_{38} \subset \mathbb{P} (1,5,6,8,19)$ & $1/120$ & $40$ & $11$ & $\heartsuit$ & $w^2 + t^4 z + t z^5$ \\
\rowcolor{lightgray}
87 & $X_{40} \subset \mathbb{P} (1,5,7,8,20)$ & $1/140$ & $40$ & $13$ & $\heartsuit$ & $t^5 + w^2$ \\
88 & $X_{42} \subset \mathbb{P} (1,1,6,14,21)$ & $1/42$ & $14$ & $11$ & $\heartsuit$ & $z^7 + t^3 + w^2$ \\
\rowcolor{lightgray}
89 & $X_{42} \subset \mathbb{P} (1,2,5,14,21)$ & $1/70$ & $14$ & $19$ & $\heartsuit$ & $t^3 + w^2$ \\
90 & $X_{42} \subset \mathbb{P} (1,3,4,14,21)$ & $1/84$ & $42$ & $7$ & $\heartsuit$ & $z^7 t + t^3 + w^2$ \\
\rowcolor{lightgray}
91 & $X_{44} \subset \mathbb{P} (1,4,5,13,22)$ & $1/130$ & $52$ & $9$ & $\heartsuit$ & $z t^3 + w^2$ \\
92 & $X_{48} \subset \mathbb{P} (1,3,5,16,24)$ & $1/120$ & $48$ & $9$ & $\heartsuit$ & $t^3 + w^2$ \\
\rowcolor{lightgray}
93 & $X_{50} \subset \mathbb{P} (1,7,8,10,25)$ & $1/280$ & $70$ & $15$ & $\heartsuit$ & $z^5 t + t^5 + w^2$ \\
94 & $X_{54} \subset \mathbb{P} (1,4,5,18,27)$ & $1/180$ & $36$ & $19$ & $\heartsuit$ & $w^2 + t^3$ \\
\rowcolor{lightgray}
95 & $X_{66} \subset \mathbb{P} (1,5,6,22,33)$ & $1/330$ & $110$ & $11$ & $\heartsuit$ & $w^2 + t^3 + z^{11}$
\end{longtable}
\endgroup

\begin{remark} \label{rem:monomWH table 1}
We explain the $7$th column ``$F (0, 0, z, t, w)$'' of Table \ref{table:Fanohyp}.
The polynomial defining $X$ is denoted by $F = F (x, y, z, t, w)$ and in the $7$th column we give an explicit description of $F (0, 0, z, t, w)$ in the following manner:
\begin{itemize}
\item If $d = 2 a_4$, then may and do assume $F (0, 0, z, t, w) = w^2 + g (z, t)$ for some homogeneous polynomial $g (z, t)$ of degree $d$.
This is possible since $w^2 \in F$ and we can kill the terms divisible by $w$ by replacing $w \mapsto w - h (z, t)$ for a suitable $h$.
\item A monomial with coefficient $1$ means that it appears in $F (0, 0, z, t, w)$ with a nonzero coefficient.
If we rescale variables $z, t$ and $w$, then we may assume that the coefficients of these monomials in $F (0, 0, z, t, w)$ are all $1$.
\item A monomial with parentheses means that its coefficient in $F (0, 0, z, t, w)$ can be $0$.
\item A pair of monomials with underline means that at least one of them appears in $F$ with nonzero coefficient.
\item For families \textnumero 17, 19, 21 and 27, the polynomial $\phi_{(i)}$ is given as follows:
\begin{itemize}
\item $\phi_{(17)} (t, w) = (w - \alpha_1 t)(w - \alpha_2 t)(w - \alpha_3 t)$.
\item $\phi_{(19)} (z, t) = (t - \alpha_1 z)(t - \alpha_2 z)(t - \alpha_3 z)(t - \alpha_4)$.
\item $\phi_{(21)} (z, t) = (t - \alpha_1 z^2)(t - \alpha_2 z^2)(t - \alpha_3 z^2)$.
\item $\phi_{(27)} (t, w) = (w - \alpha_1 t)(w - \alpha_2 t)(w - \alpha_3 t)$.
\end{itemize}
In each of the above equations, the $\alpha_i \in \mathbb{C}$ are mutually distinct.
\end{itemize}
\end{remark}

\begin{lemma} \label{lem:hyplidcclub}
Let $X = X_d \subset \mathbb P (a_0, a_1, \dots, a_4)$ be a member of a family marked $\clubsuit$ in the $6$th column of Table \ref{table:Fanohyp} and we define
\[
l_{\mathrm{ic}} := \max \{\operatorname{lcm} (a_i, a_j) \mid 0 \le i < j \le 3\},
\]
which is the positive integer given in the $4$th column of Table \ref{table:Fanohyp}.
Then, for any point $P \in \operatorname{NQsm} (X)$, there is a $P$-isolating class $l A$ for some $l \le l_{\mathrm{ic}}$.
\end{lemma}

\begin{proof}
We have $d = 2 a_4$ and $a_3 < a_4$.
This implies $w^2 \in F$.
Then the assertion follows from Lemma~\ref{lem:isoldivs}.
\end{proof}

\begin{lemma} \label{lem:hyplidcheart}
Let $X$ be a member of a family marked $\heartsuit$ in the $6$th column of Table \ref{table:Fanohyp} and we define
\[
l_{\mathrm{ic}} :=
\begin{cases}
\max \{\operatorname{lcm} (a_1, a_2), \operatorname{lcm} (a_1, a_3)\}, & \text{if $a_4 \mid d$}, \\
\max\{ \operatorname{lcm} (a_1, a_2), \operatorname{lcm} (a_1, a_3), \operatorname{lcm} (a_1, a_4)\}, & \text{if $a_4 \nmid d$}.
\end{cases}
\]
Then the non-quasismooth locus $\operatorname{NQsm} (X)$ of $X$ is contained in $U_x \cup U_y$.
For any point $P \in U_x \cup U_y$, there exists a $P$-isolating class $l A$ for some $l \le l_{\mathrm{ic}}$.
\end{lemma}

\begin{proof}
We set $f := F (0, 0, z, t, w)$ and
\begin{equation} \label{eq:WHsingxy}
\Xi := (x = y = 0) \cap \left(\frac{\partial f}{\partial z} = \frac{\partial f}{\partial t} = \frac{\partial f}{\partial w} = 0 \right) \subset X.
\end{equation}
Then $(x = y = 0) \cap \operatorname{NQsm} (X) \subset \Xi$.
The polynomial $f = f (z, t, w)$ is explicitly described in Table \ref{table:Fanohyp}.

Suppose that we are not in one of the following cases:
\begin{itemize}
\item[(i)] $X$ belongs to family \textnumero $32$ and $w z^3, t z^4 \notin F$.
\item[(ii)] $X$ belongs to family \textnumero $49$ and $t z^3 \notin F$.
\item[(iii)] $X$ belongs to family \textnumero $64$ and $t z^4 \notin F$.
\end{itemize}
In this case, it is straightforward to check that $\Xi \subset \{P_z, P_t, P_w\}$.
By assumption, $X$ is quasismooth at $P_z, P_t, P_w$ if $X$ contains $P_z, P_t, P_w$, respectively.
This shows that $X$ is quasismooth along $(x = y = 0) \cap X$ and the first assertion is proved in this case.

Suppose that we are in one of the cases (i), (ii) and (iii).
Then $X$ belongs to family $64$ and $t z^4 \notin f$.
In this case we have $\Xi = (x = y = t = 0) \subset X$ ($\Xi = (x = y = w = 0) \subset X$) if we are in the case (i) (resp.\ (ii) and (iii)), and we can write
\[
F =
\begin{cases}
\alpha x z^5 + \beta y w^2 + g, & \text{in the case (i)}, \\
\alpha x z^4 + \beta y t^3 + g, & \text{in the case (ii)}, \\
\alpha x z^5 + \beta y t^4 + g, & \text{in the case (iii)},
\end{cases}
\]
where $\alpha, \beta \in \mathbb{C}$ and $g = g (x, y, z, t, w)$ is a homogeneous polynomial that is contained in the ideal $(x, y, t)^2$ (resp.\ $(x, y, w)^2$) if we are in the case (i) (resp.\ (ii) and (iii)).
By the quasismoothness of $X$, we have $\alpha \beta \ne 0$.
This implies $(x = y = 0) \cap \operatorname{NQsm} (X) = \emptyset$ and thus $X$ is quasismooth along $(x = y = 0) \cap X$.
This proves the first assertion.

The latter assertion follows from Lemma~\ref{lem:isoldivs}.
\end{proof}

\begin{lemma} \label{lem:hyplidcspade}
Let $X$ be a member of a family marked $\spadesuit$ in the $6$th column of Table \ref{table:Fanohyp} and we define
\[
l_{\mathrm{ic}} :=
\begin{cases}
\max \{\operatorname{lcm} (a_1, a_2), \operatorname{lcm} (a_1, a_3)\}, & \text{if $a_4 \mid d$}, \\
\max\{ \operatorname{lcm} (a_1, a_2), \operatorname{lcm} (a_1, a_3), \operatorname{lcm} (a_1, a_4)\}, & \text{if $a_4 \nmid d$}.
\end{cases}
\]
Then one of the following holds.
\begin{enumerate}
\item The non-quasismooth locus $\operatorname{NQsm} (X)$ of $X$ is contained in $U_x \cup U_y$.
\item There is a point $P' \in (x = y = 0) \cap X$ such that $\operatorname{NQsm} (X) \subset U_x \cup U_y \cup \{P'\}$ and $P' \in X$ is an isolated $cA_1$ point.
\end{enumerate}
For any point $P \in U_x \cup U_y$, there exists a $P$-isolating class $l A$ for some $l \le l_{\mathrm{ic}}$.
Moreover, in the case $(2)$, the divisor class $a_2 a_3 A$ isolates $P'$.
\end{lemma}

\begin{proof}
The polynomial $f := F (0, 0, z, t, w)$ is explicitly described in the $7$th column of Table \ref{table:Fanohyp} (see also Remark \ref{rem:monomWH table 1}).

Suppose that at least one of the monomials with parentheses appear in $F$ with a nonzero coefficient.
Then we have $\Xi \subset \{P_z, P_t, P_w\}$, where $\Xi$ is as in \eqref{eq:WHsingxy}.
By the same argument as in the proof of Lemma~\ref{lem:hyplidcheart}, the assertion (1) holds in this case.

Suppose that no monomial with parentheses appears in $f$.
Then
\[
\Xi = (x = y = w = 0) \subset X.
\]
In each instance, we can write
\begin{equation} \label{eq:WHeqFXi}
F = y (z^{a_3} - t^{a_2}) + w^2 + g,
\end{equation}
where and $g = g (x, y, z, t) \in (x, y)^2$.
We see that $X$ is quasismooth along $\Xi \setminus \{P'\}$, where $P' := (0\!:\!0\!:\!1\!:\!1\!:\!0) \in \Xi$, and $X$ is not quasismooth at $P'$.
Let $\mathcal{U}$ be the open subset $(z \ne 0) \cap (t \ne 0)$ of $\mathbb{P}$ and set $U := X \cap \mathcal{U}$.
We give an explicit description of $U$.
We take positive integers $m$ and $n$ such that $m a_3 - n a_2 = 1$ and set $M = t^m/z^n$.
We set $u := z^{a_3}/t^{a_2}$, $\tilde{x} = x/M, \tilde{y} = y/M^{a_1}$ and $\tilde{w} = w/M^{a_4}$.
Note that $z/M^{a_2} = u^m$ and $t/M^{a_3} = u^n$.
Then we have
\[
\mathcal{U} = \operatorname{Spec} \mathbb{C} [\tilde{x}, \tilde{y}, \tilde{w}, u, u^{-1}] \cong \mathbb{A}^3 \times (\mathbb{A}^1 \setminus \{o\}),
\]
and $U$ is the hypersurface defined in $\mathcal{U}$ by the equation
\[
\tilde{F} (\tilde{x}, \tilde{y}, \tilde{w}, u) := F (\tilde{x}, \tilde{y}, u^m, u^n, \tilde{w}) = \tilde{y} u^{n a_2}(u - 1) + \tilde{w}^2 + \tilde{g} = 0,
\]
where $\tilde{g} = g (\tilde{x}, \tilde{y}, u^m, u^n) \in (\tilde{x}, \tilde{y})^2$.
The ponit $P' \in U$ corresponds to the point $(\tilde{x}, \tilde{y}, \tilde{w}, u) = (0, 0, 0, 1)$.
We set $\tilde{u} := u - 1$ and we make the coordinate change $u \mapsto \tilde{u}$.
Then
\[
\tilde{F}' := \tilde{F} (\tilde{x}, \tilde{y}, \tilde{w}, \tilde{u}) = \tilde{y} (\tilde{u} + 1)^{n a_2} \tilde{u} + \tilde{w}^2 + \tilde{g}',
\]
where $\tilde{g}' = \tilde{g} (\tilde{x}, \tilde{y}, \tilde{u}) \in (\tilde{x}, \tilde{y})^2$.
The point $P'$ corresponds to the origin.
The quadratic part of $\tilde{F}'$ contains the monomials $\tilde{y} \tilde{u}$ and $\tilde{w}^2$ and it is of rank at least $3$.
A Gorenstein terminal singularity is a hypersurface singularity, and it is of type $cA_k$ for some $k \ge 1$ (resp.\ $cA_1$) if and only if the quadratic part of its defining equation is of rank at least $2$ (resp.\ at least $3$).
It follows that $P' \in X$ is an isolated $cA_1$ point.
Thus we are in case (2).

The existence of a $P$-isolating divisor for $P \in U_x \cup U_y$ follows from Lemma~\ref{lem:isoldivs}.
In the case (2), the set of polynomials $\{x, y, z^{a_3} - t^{a_2}, w\}$ clearly isolates the point $P' \in X$, and hence $a_2 a_3 A$ isolates $P'$ since $a_2 a_3 > a_4$.
\end{proof}

For a family listed in Table \ref{table:Fanohyp}, the positive integer $k_{\mathrm{cA}}$ in the $5$th column is defined by
\begin{equation} \label{eq:defkcA}
k_{\mathrm{cA}} :=
\begin{dcases}
1, & \text{if $\bigstar$ is marked in the $6$th column}, \\
\lfloor \frac{4}{l_{\mathrm{ic}} (-K_X)^3} - 1 \rfloor, & \text{otherwise},
\end{dcases}
\end{equation}
where, for a real number $r$, $\lfloor r \rfloor$ is the greatest integer which is less than or equal to $r$.

\begin{proposition} \label{prop:HypexclcA}
Let $X$ be a member of a family listed in Table \ref{table:Fanohyp} satisfying the assumption of Setting \ref{set:brWCI} and let $k_{\mathrm{cA}}$ be the positive integer in the $5$th column of Table \ref{table:Fanohyp}.
Then no $cA_k$ point is a maximal center for $k \le k_{\mathrm{cA}}$.
\end{proposition}

\begin{proof}
If $X$ belongs to a family marked $\bigstar$, then the assertion follows from \cite[Proposotion 4.9]{KOPP24} and we do not treat these families.

Suppose that $P \in X$ is a $cA_k$ point that is a maximal center, where $k \le k_{\mathrm{cA}}$.
Then there exist a positive integer $n$ and a mobile linear system $\mathcal{M} \subset \left|-nK_X\right|$ such that $P$ is a center of non-canonical singularities of the pair $(X, \frac{1}{n} \mathcal{M})$.
Let $D_1, D_2 \in \mathcal{M}$ be general members.
If $X$ belong to a family marked $\spadesuit$ and $P$ is contained in $(x = y = 0) \cap X$, then $k = 1$ and we have
\begin{equation} \label{eq:HypexclcA1}
\operatorname{mult}_P (D_1 \cdot D_2) > 2 n^2
\end{equation}
by Theorem~\ref{body-MainThm}.
Otherwise we have
\begin{equation} \label{eq:HypexclcA2}
\operatorname{mult}_P (D_1 \cdot D_2) > \frac{4}{k+1} n^2 \ge \frac{4}{k_{\mathrm{cA}} + 1} n^2.
\end{equation}

By Lemmas \ref{lem:hyplidcclub}, \ref{lem:hyplidcheart} and \ref{lem:hyplidcspade}, there is a $P$-isolating class $a_2 a_3 A$ (resp.\ $l A$ for some $l \le l_{\mathrm{ic}}$) if $X$ belongs to a family marked $\spadesuit$ and $P \in (x = y = 0) \cap X$ (resp.\ otherwise).
It follows that there is an effective $\mathbb{Q}$-divisor $T \sim_{\mathbb{Q}} - a_2 a_3 K_X$ (resp.\ $T \sim_{\mathbb{Q}} - l K_X$) such that $\operatorname{ord}_P (T) \ge 1$ and that $\operatorname{Supp} (T)$ does not contain any component of $D_1 \cdot D_2$ passing through $P$ if $X$ belongs to a family marked $\spadesuit$ and $P \in (x = y = 0) \cap X$ (resp.\ otherwise).
If $X$ belongs to a family marked $\spadesuit$ and $P \in (x = y = 0) \cap X$, then by Lemma~\ref{lem:intom} and \eqref{eq:HypexclcA1} we have
\[
a_2 a_3 n^2 (-K_X^3) = (T \cdot D_1 \cdot D_2) > \operatorname{ord}_P (T) \operatorname{mult}_P (D_1 \cdot D_2) > 2 n^2,
\]
which is impossible since $a_2 a_3 (-K_X^3) \le 2$ in each instance.
Otherwise, by Lemma~\ref{lem:intom} and \eqref{eq:HypexclcA2}, we have
\[
l_{\mathrm{ic}} n^2 (-K_X^3) \ge (T \cdot D_1 \cdot D_2) \ge \operatorname{ord}_P (T) \operatorname{mult}_P (D_1 \cdot D_2) > \frac{4}{k_{\mathrm{cA}} + 1} n^2,
\]
which is impossible by \eqref{eq:defkcA}.
Thus $P \in X$ cannot be a maximal center.
\end{proof}

\subsubsection{Weighted complete intersections of codimension 2}

Let $X = X_{d_1, d_2} \subset \mathbb{P} (a_0, \dots, a_5) =: \mathbb{P}$ be a member of family listed in Table \ref{table:FanoWCI} satisfying the assumption of Setting \ref{set:brWCI}.
We assume that $a_0 \le \cdots \le a_5$ and set $A := -K_X \in \operatorname{Cl} (X)$ which is the Weil divisor class corresponding to $\mathcal{O}_X (1)$.
Let $x, y, z, t, v, w$ be the homogeneous coordinates of $\mathbb{P}$ of weights $a_0, a_1, a_2, a_3, a_4, a_5$, respectively, and we denote by $F_1 = F_1 (x, y, z, t, v, w)$ and $F_2 = F_2 (x, y, z, t, v, w)$ the homogeneous polynomials of degree $d_1$ and $d_2$, respectively, that define $X$ in $\mathbb{P}$.

\begingroup
\begin{table}[h]
\renewcommand{\arraystretch}{1.15}
\begin{center}
\caption{Fano $3$-fold WCIs of codimension $2$ and index $1$}
\label{table:FanoWCI}
\begin{tabular}{cccccc}
\hline
\textnumero & $X_{d_1, d_2} \subset \mathbb{P} (a_0, \dots, a_5)$ & $-K_X^3$ & $l_{\mathrm{ic}}$ & $k_{\mathrm{cA}}$ & Case \\
\hline
\rowcolor{lightgray}
8 & $X_{4, 6} \subset \mathbb{P} (1, 1, 2, 2, 2, 3)$ & $1$ & & $1$ & $\bigstar$ \\
14 & $X_{6, 6} \subset \mathbb{P} (1, 2, 2, 2, 3, 3)$ & $1/2$ & $2$ & $3$ & $\clubsuit_{4,5}$ \\
\rowcolor{lightgray}
20 & $X_{6, 8} \subset \mathbb{P} (1, 2, 2, 3, 3, 4)$ & $1/3$ & & $1$ & $\bigstar$ \\
24 & $X_{6, 10} \subset \mathbb{P} (1, 2, 2, 3, 4, 5)$ & $1/4$ & & $1$ & $\bigstar$ \\
\rowcolor{lightgray}
31 & $X_{8, 10} \subset \mathbb{P} (1, 2, 3, 4, 4, 5)$ & $1/6$ & & $1$ & $\bigstar$ \\
37 & $X_{8, 12} \subset \mathbb{P} (1, 2, 3, 4, 5, 6)$ & $2/15$ & & $1$ & $\bigstar$ \\
\rowcolor{lightgray}
45 & $X_{10, 12} \subset \mathbb{P} (1, 2, 4, 5, 5, 6)$ & $1/10$ & & $1$ & $\bigstar$ \\
47 & $X_{10, 12} \subset \mathbb{P} (1, 3, 4, 4, 5, 6)$ & $1/12$ & $12$ & $3$ & $\clubsuit_{4,5}$ \\
\rowcolor{lightgray}
51 & $X_{10, 14} \subset \mathbb{P} (1, 2, 4, 5, 6, 7)$ & $1/12$ & $12$ & $3$ & $\clubsuit_{3,5}$ \\
59 & $X_{12, 14} \subset \mathbb{P} (1, 4, 4, 5, 6, 7)$ & $1/20$ & $20$ & $3$ & $\clubsuit_{4,5}$ \\
\rowcolor{lightgray}
60 & $X_{12, 14} \subset \mathbb{P} (2, 3, 4, 5, 6, 7)$ & $1/30$ & $20$ & $5$ & $\clubsuit_{4,5}$ \\
64 & $X_{12, 16} \subset \mathbb{P} (1, 2, 5, 6, 7, 8)$ & $2/35$ & & $1$ & $\bigstar$ \\
\rowcolor{lightgray}
71 & $X_{14, 16} \subset \mathbb{P} (1, 4, 5, 6, 7, 8)$ & $1/30$ & $30$ & $3$ & $\heartsuit_{4,5}$ \\
75 & $X_{14, 18} \subset \mathbb{P} (1, 2, 6, 7, 8, 9)$ & $1/24$ & $24$ & $3$ & $\heartsuit_{3,5}$ \\
\rowcolor{lightgray}
76 & $X_{12, 20} \subset \mathbb{P} (1, 4, 5, 6, 7, 10)$ & $1/35$ & $35$ & $3$ & $\heartsuit_{3,5}$ \\
78 & $X_{16, 18} \subset \mathbb{P} (1, 4, 6, 7, 8, 9)$ & $1/42$ & $42$ & $3$ & $\heartsuit_{4,5}$ \\
\rowcolor{lightgray}
84 & $X_{18, 30} \subset \mathbb{P} (1, 6, 8, 9, 10, 15)$ & $1/120$ & $40$ & $11$ & $\heartsuit_{3,5}$ \\
85 & $X_{24, 30} \subset \mathbb{P} (1, 8, 9, 10, 12, 15)$ & $1/180$ & $90$ & $7$ & $\heartsuit_{4,5}$
\end{tabular}
\end{center}
\end{table}
\endgroup

\begin{lemma} \label{lem:wcilidcclub}
Let $X$ be a member of a family marked  $\clubsuit_{m, 5}$ for some $m \in \{3, 4\}$ in the $6$th column of Table \ref{table:FanoWCI} and we define
\[
l_{\mathrm{ic}} := \max \{\, \operatorname{lcm} (a_i, a_j) \mid i, j \in \{0, \dots, 4\} \setminus \{m\}\,\}.
\]
Then, for any point $P \in \operatorname{NQsm} (X)$, there is a $P$-isolating class $lA$ for some $l \le l_{\mathrm{ic}}$.
\end{lemma}

\begin{proof}
We set $s := t$ (resp.\ $s := v$) if $m = 3$ (resp.\ $m = 4$.
Then it is easy to see that
\[
\bigcap_{u \in \{x, y, z, t\} \setminus \{s\}} (u = 0) \cap X = \emptyset.
\]
The assertion follows by applying (1-c) of Lemma~\ref{lem:isoldivs} for $(m_1, m_2) = (m, 5)$.
\end{proof}

\begin{lemma} \label{lem:wcilidcheart}
Let $X$ be a member of a family marked  $\heartsuit_{m, 5}$ for some $m \in \{3, 4\}$ in the 6th column of Table \ref{table:FanoWCI} and we define
\[
l_{\mathrm{ic}} := \max \{\, \operatorname{lcm} (a_r, a_j) \mid r \in \{1, 2\}, j \in \{0, \dots, 4\} \setminus \{m\}\,\}.
\]
Then the non-quasismooth locus $\operatorname{NQsm} (X)$ of $X$ is contained in $U_x \cup U_y$.
For any point $P \in U_x \cup U_y$, there exists a $P$-isolating class $l A$ for some $l \le l_{\mathrm{ic}}$.
\end{lemma}

\begin{proof}
We consider the matrix
\[
J :=
\begin{pmatrix}
\partial f_1/\partial z & \partial f_1/\partial t & \partial f_1/\partial v & \partial f_1/\partial w \\
\partial f_2/\partial z & \partial f_2/\partial t & \partial f_2/\partial v & \partial f_2/\partial w
\end{pmatrix}
\]
and then define
\[
\Xi := \{\, P \in (x = y = 0) \cap X \mid \operatorname{rank} J (P) < 2\,\}.
\]
We see that $\operatorname{NQsm} (X) \cap (x = y = 0) \subset \Xi$.

By a straightforward computation in each instance, we can easily compute the set $\Xi$ and it is described in the 5th column of Table \ref{table:WCIeq}.

Suppose that $X$ belongs to the family \textnumero $71$ and $t z^2 \notin F_2$.
In this case $\Xi = (x = y = w = 0)$.
We have $z^3 x \in F_2$ (see Remark \ref{rem:monomWH table 2}) and we may assume that its coefficient is $1$ by rescaling $x$.
Then we can write
\[
\begin{split}
F_1 &= \alpha y z^2 + w t + g_1, \\
F_2 &= x z^3 + \beta t^2 y + g_2,
\end{split}
\]
where $\alpha, \beta \in \mathbb{C}$ and $g_i = g_i (x, y, z, t, w)$ is a homogeneous polynomials that is contained in the ideal $(x, y, v, w)^2$.
By the quasismoothness of $X$ at $P_z$ and $P_t$, we have $\alpha \ne 0$ and $\beta \ne 0$.
Then we see that the Jacob matrix of $X$ is of rank $2$ along the set $(z \ne 0) \cap (x = y = 0) \cap X$.
It follows that $\operatorname{NQsm} (X) \cap (x = y = 0) \subset \{P_t\}$.

In the other instances, we have $\Xi \subset X \cap \operatorname{Sing} (\mathbb{P})$, and hence we have $\operatorname{NQsm} (X) \subset X \cap \operatorname{Sing} (\mathbb{P})$ in all instances.
This proves the first assertion since $X$ is quasismooth along $X \cap \operatorname{Sing} (\mathbb{P})$.

We set $s := t$ (resp.\ $s := v$) if $m = 3$ (resp.\ $m = 4$).
Then it is easy to see that
\[
\bigcap_{u \in \{x, y, z, t\} \setminus \{s\}} (u = 0) \cap X = \emptyset.
\]
The second assertion follows by applying (2-c) of Lemma~\ref{lem:isoldivs} for $(m_1, m_2) = (m, 5)$.
\end{proof}

\begin{table}[h]
\caption{Equations for codimension $2$ WCIs belonging to families marked $\heartsuit$}
\label{table:WCIeq}
\centering
\begin{tabular}{cllccl}
\toprule
\textnumero & $F_1 (0,0,z,t,v,w)$ & $F_2 (0,0,z,t,v,w)$ & Case & $\Xi$ & Remark \\
\midrule
$71$ & $w t + v^2$ & $w^2 + t z^2$ & $t z^2 \in F_2$ & $\{P_z, P_t\}$ & \\
\cmidrule(lr){2-6}
& $w t + v^2$ & $w^2$ & $t z^2 \notin F_2$ & $(x = y = v = w = 0)$ & $z^3 x \in F_2$ \\
\cmidrule{1-6}
$75$ & $v z + t^2$ & $w^2 + z^3$ & & $\{P_v\}$ & \\
\cmidrule{1-6}
$76$ & $v z + t^2$ & $w^2 + v^2 t + z^4$ & & $\{P_v\}$ & \\
\cmidrule{1-6}
$78$ & $w t + v^2$ & $w^2 + z^3$ & & $\{P_t\}$ & \\
\cmidrule{1-6}
$84$ & $v z + t^2$ & $w^2 + v^3$ & & $\{P_z\}$ & \\
\cmidrule{1-6}
$85$ & $w z + v^2$ & $w^2 + v z^2 + t^3$ & & $\{P_z\}$ & \\
\bottomrule
\end{tabular}
\end{table}

\begin{remark} \label{rem:monomWH table 2}
We explain the equations given in Table \ref{table:WCIeq} in more detail.
In the $2$nd and $3$rd columns we give an explicit descriptions of $F_1 (0, 0, z, t, v, w)$ and $F_2 (0, 0, z, t, v, w)$, respectively, in the following manner:
\begin{itemize}
\item We have $d_2 = 2 a_5$ and we have $w^2 \in F_2$ by the quasismoothness assumption.
We may and do assume that $F (0, 0, z, t, v, w) = w^2 + g (z, t, v)$ for some homogeneous polynomial $g (z, t, v)$ of degree $d_2$.
This is possible since we can kill the terms divisible by $w$ by replacing $w \mapsto w - h (z, t, v)$ for a suitable $h$.

\item For family \textnumero $71$, we have case divisions as follows:
We are in the first case if $t z^2 \in F_1$, otherwise we are in the second case.
In the second case, we have $z^3 x \in F_2$ by the quasismoothness assumption.
\item If we rescale variables $z, t$ and $w$, then we may assume that the coefficients of these monomials in $F (0, 0, z, t, w)$ are all $1$.
\end{itemize}
\end{remark}

For a family listed in Table \ref{table:FanoWCI}, the positive integer $k_{\mathrm{cA}}$ in the $5$th column is defined by
\begin{equation} \label{eq:defkcAWCI}
k_{\mathrm{cA}} :=
\begin{dcases}
1, & \text{if $\bigstar$ is marked in the $6$th column}, \\
\lfloor \frac{4}{l_{\mathrm{ic}} (-K_X)^3} - 1 \rfloor, & \text{otherwise}.
\end{dcases}
\end{equation}

\begin{proposition} \label{prop:WCIexclcA}
Let $X$ be a member of a family listed in Table \ref{table:FanoWCI} satisfying the assumption of Setting \ref{set:brWCI} and let $k_{\mathrm{cA}}$ be the positive integer in the $5$th column of Table \ref{table:FanoWCI}.
Then no $cA_k$ point is a maximal center for $k \le k_{\mathrm{cA}}$.
\end{proposition}

\begin{proof}
The proof is identical to that of Proposition~\ref{prop:HypexclcA}.
If $X$ belongs to a family marked $\bigstar$, then the assertion follows from \cite[Proposotion 4.9]{KOPP24} and we do not treat these families.

Suppose that $P \in X$ is a $cA_k$ point that is a maximal center, where $k \le k_{\mathrm{cA}}$.
Then there exist a positive integer $n$ and a mobile linear system $\mathcal{M} \subset \left|-nK_X\right|$ such that $P$ is a center of non-canonical singularities of the pair $(X, \frac{1}{n} \mathcal{M})$.
Let $D_1, D_2 \in \mathcal{M}$ be general members.
By Lemmas \ref{lem:wcilidcclub}, and \ref{lem:wcilidcheart}, there is a $P$-isolating class $l A$ for some $l \le l_{\mathrm{ic}}$.
It follows that there is an effective $\mathbb{Q}$-divisor $T \sim_{\mathbb{Q}} - l K_X$ such that $\operatorname{ord}_P (T) \ge 1$ and that $\operatorname{Supp} (T)$ does not contain any component of $D_1 \cdot D_2$ passing through $P$.
We have
\[
l_{\mathrm{ic}} n^2 (-K_X^3) \ge l n^2 (-K_X^3) =(T \cdot D_1 \cdot D_2) \ge \operatorname{ord}_P (T) \operatorname{mult}_P (D_1 \cdot D_2) > \frac{4}{k_{\mathrm{cA}} + 1} n^2,
\]
where the second inequality follows from Lemma~\ref{lem:intom} and the third inequality follows from Theorem~\ref{body-MainThm}.
This is impossible by \eqref{eq:defkcAWCI}.
Thus $P \in X$ cannot be a maximal center.
\end{proof}

\subsubsection{Completion of Proof of Theorem~\ref{thm:brWCI}}

Let $X$ be as in Theorem~\ref{thm:brWCI}.
Then, by \cite[Propositions 5.5 and 5.12]{KOPP24}, no curve and no smooth point is a maximal center of $X$.
Moreover, by \cite[Proposition 5.15]{KOPP24}, for each terminal quotient singular point $P \in X$, either it is not a maximal center, or there is an elementary self-link initiated by the Kawamata blow-up of $X$ at $P$.
Note that these results are due to \cite{CP17} as it is explained in \cite{KOPP24}.
Finally, by Proposition~\ref{prop:HypexclcA} and \ref{prop:WCIexclcA}, no $cA_k$ point is a maximal center for $k \le k_{\mathrm{cA}}$.
By Theorem~\ref{thm:chractBR}, $X$ is birationally rigid and Theorem~\ref{thm:brWCI} is proved.

\providecommand{\bysame}{\leavevmode\hbox to3em{\hrulefill}\thinspace}
\providecommand{\MR}{\relax\ifhmode\unskip\space\fi MR }
\providecommand{\MRhref}[2]{%
  \href{http://www.ams.org/mathscinet-getitem?mr=#1}{#2}
}
\providecommand{\href}[2]{#2}

\vspace{0.5\baselineskip}
\ShowAffiliations{\\[1\baselineskip]}%

\end{document}